\documentclass[11pt]{amsart}

\usepackage{geometry}
\usepackage[all]{xy}
\usepackage[toc,page]{appendix}
\usepackage{epstopdf}
\usepackage{amssymb,amsmath,amsfonts,amsthm,graphicx,enumerate,amscd}
\usepackage{mathrsfs}

\theoremstyle{plain}
\newtheorem{theorem}{Theorem}[section]

\newtheorem{corollary}[theorem]{Corollary}
\newtheorem{lemma}[theorem]{Lemma}
\newtheorem{proposition}[theorem]{Proposition}

\theoremstyle{definition}
\newtheorem{example}[theorem]{Example}
\newtheorem{remark}[theorem]{Remark}

\newtheorem{definition}[theorem]{Definition}

\allowdisplaybreaks
\title{Equivariant twisted Real $K$-theory of compact Lie groups}
\author{Chi-Kwong Fok}

\begin{document}
\maketitle
\begin{abstract}
	Let $G$ be a compact, connected, and simply-connected Lie group viewed as a $G$-space via the conjugation action. The Freed-Hopkins-Teleman Theorem (FHT) asserts a canonical link between the equivariant twisted $K$-homology of $G$ and its Verlinde algebra. In this paper, we give a generalization of FHT in the presence of a Real structure of $G$. Along the way we develop preliminary materials necessary for this generalization, which are of independent interest in their own right. These include the definitions of Real Dixmier-Douady bundles, the Real third cohomology group which is shown to classify the former, and Real $\text{Spin}^c$ structures.
\end{abstract}
	
\tableofcontents
\section{Introduction}
Let $G$ be a compact connected Lie group. The Freed-Hopkins-Teleman Theorem (FHT) is a recent deep result which asserts a canonical isomorphism between the twisted equivariant $K$-homology of $G$ with its Verlinde algebra, the representation group of positive energy representations of the loop group $LG$ equipped with the intricately defined ring structure called the fusion product (cf. \cite{Fr}, \cite{FHT1}, \cite{FHT2}, \cite{FHT3}, Theorem \ref{FHT}). The Verlinde algebra is an object of great interest in mathematical physics and algebraic geometry. One of the remarkable aspects of the Freed-Hopkins-Teleman Theorem is that it provides an algebro-topological approach to interpreting the fusion product, which is usually defined using conformal blocks or moduli spaces of $G$-bundles on Riemann surfaces (cf. \cite{Be}, \cite{BL} and \cite{V}). Moreover, the Freed-Hopkins-Teleman Theorem also provides the framework for a formulation of geometric quantization of q-Hamiltonian spaces (cf. \cite{M2} and \cite{M3}).

Motivated by the index theory of real elliptic operators, Atiyah introduced $KR$-theory in \cite{At}. It is a version of topological $K$-theory for the category of Real spaces, i.e. topological spaces with involutions (see Appendix for definition and basic properties), and used by Atiyah to derive the 8-periodicity of $KO$-theory from the 2-periodicity of complex $K$-theory. $KR$-theory can be regarded as a unifying thread of $KO$-theory, complex $K$-theory and $\mathit{KSC}$-theory (cf. \cite{At}). In recent years there is a rekindled interest in $KR$-theory, motivated by its applications in mathematical physics, as it classifies the D-brane charges in orientifold string theory (see e.g. \cite{DFM} and \cite{DMR}) and provides invariants for topological phases and finer classification of topological insulators (\cite{FM}). It is due to these fascinating developments and the possible applications in those areas in mathematical physics that 
%as well as moduli spaces of nonorientable surfaces, 
we find it is of interest to obtain a generalization of FHT in the context of $KR$-theory. For simplicity, we will only work on the case where $G$ is simple and simply-connected and the twists in $KR$-theory are ungraded. Note that FHT deals with more general compact Lie groups and $\mathbb{Z}_2$-graded twists. 

Our main result shows that, by incorporating a group anti-involution of $G$ (i.e. composition of group inversion and an automorphic involution) and applying a structure theorem for twisted $KR$-homology (theory), we have that the corresponding equivariant twisted $KR$-homology of $G$ is essentially a module over the equivariant $KR$-homology coefficient ring, generated by the irreducible positive energy representations of real, complex and quaternionic types (cf. Corollary \ref{realFHTstrthm}). Moreover, the ring structure of the equivariant twisted $KR$-homology induced by a certain modified Pontryagin product, when restricted to those generators of positive energy representations, is precisely the fusion product. In particular, the degree zero piece of the equivariant Real twisted $KR$-homology of $G$ gives the Real Verlinde algebra, the Grothendieck group of the isomorphism classes of Real positive energy representations of the Real loop group $LG$, where the involution is induced by the Lie group involution on $G$ and reflection on the loop. This is the content of Theorem \ref{mainthm0}. In short, the group anti-involution as the additional Real structure in the equivariant twisted $KR$-homology respects the algebra structure of the Verlinde algebra. We also describe the ring structure of the equivariant twisted $KR$-homology of $G$ as the quotient of the $KR$-homology coefficient ring by what we call the Real Verlinde ideal, in parallel to the description of the Verlinde algebra as the quotient of the representation ring by the Verlinde ideal (cf. Theorem \ref{mainthm}).

The organization of this paper is as follows. In Section \ref{twistedkhomo}, we review equivariant $K$-homology and incorporate such additional structures as involutions and twists (i.e. local coefficients) to define equivariant twisted $KR$-theory (homology), which are necessary for the generalization of FHT in the Real context.  We define Real Dixmier-Douady bundles which are used to realize the twists of $KR$-theory (homology) and the Real third equivariant cohomology group which is shown to classify the twists (real twists, their classification and twisted $KR$-theory were studied in the groupoid context in \cite{Mo} and \cite{Mo2}, but the Real third cohomology group used there is not the same as ours. See Remark \ref{diffmo}). We show that when the involution is trivial, this Real third equivariant cohomology group classifies the real twists for $KO$-theory, but with extra information analoguous to what the Frobenious-Schur indicator gives (Example \ref{FrobSchur}). We also develop the notion of Real $\text{Spin}^c$ structures, which are the $KR$-theoretic orientation, and which we utilize to conclude that $KH$-theory (Quaternionic $K$-theory, see Definition \ref{QuatK}) can be viewed as twisted $KR$-theory.

Section \ref{eqfundDD} concerns the computation of the group of the equivariant Real Dixmier-Douady classes of $G$ equipped with an anti-involution, and the construction of an equivariant Real fundamental Dixmier-Douady bundle (cf. Definition \ref{eqrealfundDD}) following the idea in \cite{M1}. We find that equipping $G$ with an anti-involution rather than an automorphic involution enable it to support an equivariant Real fundamental Dixmier-Douady bundle, making it the appropriate candidate for formulating the Real FHT. In Section \ref{realFHT} we put together the preliminary material in previous sections, define the modified Pontryagin product on the equivariant twisted Real $K$-homology of $G$ and prove the main results Corollary \ref{realFHTstrthm}, Theorems \ref{mainthm0} and \ref{mainthm}. The Appendix reviews the definitions and basic properties of $KR$-theory and Real representation rings.

\textbf{Acknowledgment}. The author would like to thank Reyer Sjamaar for suggesting the problem of generalizing FHT in the Real context and his patient guidance and support in the course of writing this paper, Eckhard Meinrenken for his suggestions for improving the presentation of Section \ref{MorG}, El-ka\"ioum M. Moutuou for pointing out the relevance of his work \cite{Mo} and \cite{Mo2}, and Mathai Varghese for improving the exposition of the paper. He is also grateful for the referee's critical comments on the paper. 

\section{Notations, definitions and conventions}\label{notedefs}
Throughout this paper, we let $G$ be a compact, simple and simply-connected Lie group. We use $\Gamma=\{1, \gamma\}$ to denote the group of order 2, and $\sigma_X$ to denote the involutive homeomorphism induced by $\gamma$ on the Real space $X$. If $X$ is a group $G$, $\sigma_G$ always means an involutive automorphism, while $a_G$ means the corresponding anti-involution $a_G$, where $\text{inv}$ means group inversion. We sometimes suppress the notations for involution if there is no danger of confusion about the involutive homeomorphism. For example, we simply use $G$ with the understanding that it is equipped with $\sigma_G$, while we use $G^-$ to mean the group $G$ equipped with $a_G$. By abuse of notation, we denote by $G^-$ the group $G$ viewed as a $G\rtimes\Gamma$-space, where the first factor acts by conjugation while the second acts by the anti-involution $a_G$. Following the convention in \cite{At}, we use the word `Real' (with a capital R) in all contexts involving involutions, so as to avoid confusion with the word `real' with the usual meaning. For example, `Real $K$-theory' is used interchangeably with $KR$-theory, whereas `real $K$-theory' means $KO$-theory. Moreover, we do not discriminate between the meanings of the terms `Real structure' and `involution'. 

We let $T$ be a fixed choice of maximal torus, and $W$ the corresponding Weyl group. We let $\Lambda\subset\mathfrak{t}$ be the coroot lattice and $\Lambda^*\subset\mathfrak{t}^*$ the weight lattice. We fix a choice of simple roots $\{\alpha_1, \cdots, \alpha_l\}$ and the positive Weyl chamber $\mathfrak{t}^*_+\subset \mathfrak{t}^*$, and adopt the convention that $\alpha_0=-\alpha_{\text{max}}$, where $\alpha_\text{max}$ is the highest root. We let $B$ be the basic inner product on $\mathfrak{g}^*$, which is the bi-invariant inner product such that $B(\alpha_\text{max}, \alpha_\text{max})=2$. $B$ is used to identify $\mathfrak{t}$ and $\mathfrak{t}^*$. We let $B^\flat: \mathfrak{t}\to\mathfrak{t}^*$ and $B^\sharp: \mathfrak{t}^*\to \mathfrak{t}$. The dual Coxeter number, $\textsf{h}^\vee$, of $G$ is then defined to be $1+B(\rho, \alpha_\text{max})$, where $\rho$ is the half sum of positive roots. We let $\Delta^k$ be the level $k$ closed Weyl alcove of $\mathfrak{t}_+^*$, defined by the inequalities
\[\lambda(\alpha_i^\vee)+k\delta_{i, 0}\geq 0\text{ for }i=0, \cdots, l\]
with vertices labelled by $\{0, \cdots, l\}$, so that the origin is labelled 0. We use $\Delta$ to denote the ordinary closed Weyl alcove $\Delta^1$. Let $\Lambda_k^*$ be $\Delta^k\cap\Lambda^*$, the level $k$ weights. 

By abuse of notation, we also use $\Delta$ to denote $B^\sharp(\Delta)\subset \mathfrak{t}$. Let $I\subseteq\{0, \cdots, l\}$. Let $\Delta_I$ be the closed subsimplex of $\Delta$ spanned by vertices with labels from the index set $I$. Let $W_I$ be the subgroup of $W$ fixing $\Delta_I$. We also let $G_I$ be the stabilizer subgroup of $\Delta_I$ and $\Lambda_I$, the coroot lattice of $G_I$. 

%Define
%\[W_I^k=\begin{cases}&\text{the group generated by reflections across ker}(\alpha_i^\vee)\text{ for }i\notin I,\text{ if }0\in I\\ 
%&\text{the group generated by reflections across ker}(\alpha_i^\vee)\text{ for }i\neq 0\text{ and }i\notin I, \\
%&\text{and ker}(\alpha_0^\vee)-\frac{k\alpha_0}{2},\text{ if }0\notin I
%\end{cases}\]

If we view $R(G)$ as the ring of characters of $G$, then the level $k$ Verlinde ideal $I_k$ can be defined as the vanishing ideal of 
\[\left.\left\{\text{exp}_T\left(B^\sharp\left(\frac{\lambda+\rho}{k+\textsf{h}^\vee}\right)\right)\right|\lambda\in\Lambda_k^*\right\}\]
and the level $k$ Verlinde algebra can be alternatively defined as $R_k(G):=R(G)/I_k$ (cf. \cite{Be}). 

\section{Twisted $KR$-homology}\label{twistedkhomo}
This Section is devoted to background material on (analytic) $K$-homology, its equivariant, Real and twisted version, the classification of (Real) twists (which are realized by Dixmier-Douady bundles in this paper), and the notion of Real $\text{Spin}^c$ structures. We refer the reader to Appendix or \cite{At}, \cite{AS} and \cite{F} for the definitions and basic properties of $KR$-theory, and \cite{Mo} and \cite{Mo2} for Real twists of groupoids, their classification (which is different from ours, see Remark \ref{diffmo}) and the construction of twisted $KR$-theory in the context of groupoids. For (non-equivariant) twisted $KR$-homology defined using bundle gerbes and geometric cycles, see \cite{HMSV}.

\subsection{(Real) $K$-homology}
$K$-homology is a homology theory dual to $K$-theory through the $K$-theory version of Poincar\'e duality, where a manifold is oriented in $K$-theory if it has a $\text{Spin}^c$ structure. Inspired by the Atiyah-Singer index theorem, which he used to realize the $K$-theoretic Poincar\'e duality, Kasparov gave the first definition of $K$-homology (cf. \cite{Kas}). In this Section and the next, we follow Kasparov's definition, give a quick review of $K$-homology and introduce its Real and twisted version. Most of the materials in this Section are directly taken from \cite{BHS}, \cite{M1} and \cite{HR}, to the latter of which we refer the reader for more details on the subject.

\begin{definition}
	Let $\textsf{A}$ be a separable $\mathbb{Z}_2$-graded $G-C^*$-algebra. A (zero-graded) \emph{Fredholm module} over $\textsf{A}$ is a triple $(\rho, \mathcal{H}, F)$ where
	\begin{enumerate}
		\item $\mathcal{H}=\mathcal{H}^+\oplus\mathcal{H}^-$ is a separable $\mathbb{Z}_2$-graded $G$-Hilbert space, 
		\item $\rho: \textsf{A}\to B(\mathcal{H})$ is a representation of $\textsf{A}$ by bounded linear operators on $\mathcal{H}$ which preserves the grading and is $G$-equivariant, and
		\item $F$ is an odd graded $G$-equivariant operator on $\mathcal{H}$ such that
		\[(F^2-1)\rho(a)\sim 0, (F-F^*)\rho(a)\sim 0, [F, \rho(a)]\sim 0\]
		for all $a\in\textsf{A}$, where $\sim$ means equality modulo compact operators.
	\end{enumerate}
	A $q$-graded Fredholm module is one equipped with odd-graded skew adjoint operators $\varepsilon_1, \cdots, \varepsilon_q$ on $\mathcal{H}$, called the grading operators, such that
	\[\varepsilon_i^2=-\text{Id}, \ \ \varepsilon_i\varepsilon_j=-\varepsilon_j\varepsilon_i\text{ for }i\neq j\]
	and such that they commute with $F$ and $\rho(a)$ for all $a\in \textsf{A}$.
\end{definition}
\begin{definition}
	\begin{enumerate}
		\item We say the two Fredholm modules $(\rho, \mathcal{H}, F)$ and $(\rho', \mathcal{H}', F')$ are \emph{unitarily equivalent} if there is a degree zero unitary isomorphism $U: \mathcal{H}'\to \mathcal{H}$ such that 
		\[(\rho', \mathcal{H}', F')=(U^*\rho U, \mathcal{H}', U^*FU)\]
		Two $q$-graded Fredholm modules are unitarily equivalent if there exists a unitary isomorphism as above which in addition intertwines with the grading operators $\varepsilon_i$ and $\varepsilon'_i$.
		\item Two Fredholm modules $(\rho, H, F_0)$ and $(\rho, H, F_1)$ are \emph{operator homotopic} if there exists a norm continuous function $t\mapsto F_t$ for $t\in[0, 1]$ such that $(\rho, \mathcal{H}, F_t)$ is a Fredholm module for $t\in [0, 1]$.
	\end{enumerate}
\end{definition}
\begin{definition}[Kasparov's $K$-homology]\label{defkhom} The equivariant $K$-homology group of a $\mathbb{Z}_2$-graded $G-C^*$-algebra $\textsf{A}$, $K_G^0(\textsf{A})$, is the quotient of the abelian group with one generator $[x]$ for each unitary equivalent class of Fredholm modules over $A$ by the following relations
	\begin{enumerate}
		\item if $x_0$ and $x_1$ are operator homotopic then $[x_0]=[x_1]$ in $K^0_G(\textsf{A})$, and
		\item $[x_0\oplus x_1]=[x_0]+[x_1].$
	\end{enumerate}
	We define $K^{-q}_G(\textsf{A})$ for $q\geq 0$ similarly using $q$-graded Fredholm modules, and for $q<0$ using Bott periodicity. If $\textsf{A}$ is an ungraded $G-C^*$-algebra, its $K$-homology is defined to be the one where $\textsf{A}$ is the even part and 0 the odd part. 
\end{definition}

\begin{remark}
	One may, by regarding the grading operators as the canonical generators of the (complex) Clifford algebra, equivalently define $K_G^{-q}(\textsf{A})$ to be $K_G^0(\textsf{A}\widehat{\otimes}\mathbb{C}\text{l}_q)$.
\end{remark}
	Like $K$-theory, $K$-homology is also 2-periodic (cf. \cite[Proposition 8.2.13]{HR}). The use of superscripts to denote the grading of $K$-homology of $G$-$C^*$-algebras is due to the fact that the $K$-homology functor on the category of $G$-$C^*$-algebras is contravariant. For if $(\rho, \mathcal{H}, F)$ is a Fredholm module over $\textsf{A}$, and $\alpha: \textsf{A}'\to\textsf{A}$ is a $G$-$C^*$-algebra homomorphism, then $(\rho\circ\alpha, \mathcal{H}, F)$ is a Fredholm module over $\textsf{A}'$. 
	
\begin{definition}
	The equivariant $K$-homology of a locally compact $G$-space $X$ is defined to be the equivariant $K$-homology of the $G-C^*$-algebra of space of continuous functions on $X$ vanishing at infinity.
		\[K_q^G(X):=K_G^{-q}(C_0(X)).\]
\end{definition}
The $K$-homology functor on the category of topological spaces is covariant, as opposed to the contravariance of its counterpart on the category of $G-C^*$-algebras, hence the use of subscripts to denote the grading. 

\begin{example}
	Let $X=\text{pt}$ equipped with the trivial $G$-action. Then $C(X)$ is isomorphic to $\mathbb{C}$ with trivial $G$-action. There is a unique representation $\rho$ of $\mathbb{C}$ by bounded linear operators on the $G$-Hilbert space $\mathcal{H}$, namely the scalar multiplication. It follows that the odd-graded operator $F$ in any Fredholm module over $C(X)$ is a $G$-Fredholm operator on $\mathcal{H}$ by definition. The $K$-homology class of any given Fredholm module is then determined by the $G$-Fredholm index of $F$, which is a (finite-dimensional) representation of $G$. Hence $K_0^G(\text{pt})\cong R(G)$. It can also be shown easily that $K_{1}^G(\text{pt})=0$.
\end{example}

\begin{definition}
	Recall that $\text{Spin}^c(n)=\text{Spin}(n)\times_{\mathbb{Z}_2}U(1)$, where $\mathbb{Z}_2$ acts on the central circle $U(1)$ by multiplication by $-1$. Let $V$ be a rank $n$ orientable Euclidean $G$-vector bundle over $M$. It is $G$-$\text{Spin}^c$ if there exists, over $M$, an equivariant $G$-principal $\text{Spin}^c(n)$-bundle $\widetilde{P}$ whose structure group lifts that of the  oriented frame bundle $P$ $G$-equivariantly. Equivalently, $V$ is $G$-$\text{Spin}^c$ if $\bigwedge^*V\otimes_\mathbb{R}\mathbb{C}$ is $G$-equivariantly isomorphic to $\widetilde{P}\times_{\text{Spin(n)}^c}\mathbb{C}\text{l}_n$.
\end{definition}
\begin{example}\label{fundclass}
	Let $M$ be an $n$-dimensional oriented Riemannian $G$-manifold. It is $G$-$\text{Spin}^c$ if $TM$ is. Let $S:=\bigwedge^*TM\otimes_\mathbb{R}\mathbb{C}=\widetilde{P}\times_{\text{Spin(n)}^c}\mathbb{C}\text{l}_n$ be the spinor bundle, which comes equipped with the Clifford multiplication of $\mathbb{C}\text{l}(TM)$ on the left and the Clifford multiplication of $\mathbb{C}\text{l}_n$ on the right. The \emph{fundamental class} $[M]\in K^G_n(M)$ corresponding to a given $G$-$\text{Spin}^c$ structure of $M$ is given by the $n$-graded Fredholm module $(\rho, \mathcal{H}, F)$, where
	\begin{enumerate}
		\item $\mathcal{H}=L^2(M, S)$, graded by the even and odd degree parts of $S$, and the $n$ grading operators are right multiplication by the $n$ generators by $\mathbb{C}\text{l}_n$,
		\item $\rho$ is the representation of $C_0(M)$ on $\mathcal{H}$ by left multiplication, and 
		\item $F$ is the Dirac operator $D$ suitably normalized so that it becomes a bounded operator.
	\end{enumerate}
	The Poincar\'e duality pairing between $K$-homology and $K$-theory is realized by the index pairing, which produces the Fredholm index of a certain operator on a Hilbert space constructed out of the Fredholm module and the vector bundle representing the relevant $K$-homology and $K$-theory classes respectively. For instance, for an oriented $G$-$\text{Spin}^c$ Riemannian manifold $M$ of dimension $n$, the index pairing of $[M]$ and the $K$-class $[E]\in K_G^0(M)$ of a $G$- vector bundle $E$ gives the (analytic) $G$-index of the twisted Dirac operator $D_E$ on $S\otimes E$ (cf. \cite[Lemma 11.4.1]{HR}). This pairing with the fundamental class is tantamount to the wrong way map in $K$-theory induced by the map collapsing $M$ to a point. The $K$-theoretic Poincar\'e duality asserts that 
	\begin{align*}
		K^*_G(M)&\to K_{n-*}^G(M)\\
		[E]&\mapsto [(\rho, L^2(M, S\otimes E), D_E)]
	\end{align*}
	is an isomorphism. 
\end{example}
	
	As one would expect, equivariant $KR$-homology can be defined by adding natural Real structures throughout the definition of equivariant $K$-homology (cf. \cite[Appendix B]{HR} for a brief discussion on the set-up of $KR$-homology). For the convenience of the reader, we shall spell out the definitions of $KR$-homology. 
	\begin{definition}
		$\textsf{A}$ is a separable $\mathbb{Z}_2$-graded Real $G-C^*$-algebra if it is a graded $G-C^*$-algebra with an anti-linear involution $\sigma_\textsf{A}$ of degree zero such that 
		\begin{eqnarray}\label{calgcomp}
			\sigma_\textsf{A}(g\cdot a)=\sigma_G(g)\cdot\sigma_\textsf{A}(a)
		\end{eqnarray} 
		where $\sigma_G$ is a group involution on $G$. A Real Fredholm module over $\textsf{A}$ is one where $\mathcal{H}$ is equipped with an anti-linear zero-graded involution $\sigma_\mathcal{H}$ satisfying a compatibility relation with the $G$-action similar to Equation (\ref{calgcomp}), $\rho$ intertwines with $\sigma_\textsf{A}$ and $\sigma_\mathcal{H}$, and $F$ commutes with $\sigma_\mathcal{H}$. A $(p, q)$-graded Real Fredholm module satisfies an additional condition that the $p+q$ grading operators commute with $\sigma_\mathcal{H}$ and satisfy
		\[\varepsilon_i^2=\begin{cases}-\text{Id}\ \ &\text{ if }1\leq i\leq q\\ \text{Id}\ \ &\text{ if }q+1\leq i\leq p+q\end{cases}\text{ and }\varepsilon_i\varepsilon_j=-\varepsilon_j\varepsilon_i\text{ for }i\neq j\]
		We simply call a $(0, q)$-graded Real Fredholm module $q$-graded Real Fredholm module. The notions of unitary equivalence and operator homotopy of Real Fredholm modules, and hence equivariant $KR$-homology $KR^{p, q}_G(\textsf{A})$ (resp. $KR^{-q}_G(\textsf{A})$), can be defined straightforwardly by requiring that $\sigma_\mathcal{H}$ respect various structures in Definition \ref{defkhom} and using $(p, q)$-graded Real Fredholm modules (resp. $q$-graded Real Fredholm modules). Define, for a Real locally compact $G$-space $X$, its equivariant $KR$-homology $KR^G_{q, p}(X)$ (resp. $KR^G_q(X)$) to be $KR^{p, q}_G(C_0(X))$ (resp. $KR^{-q}_G(C_0(X))$).
	\end{definition}
	We shall point out notable changes brought by the addition of Real structures. First, the $KR$-homology is (1, 1)-periodic if double grading is used, and 8-periodic if single grading is used. Repeated application of $(1, 1)$-periodicity yields $KR^{p, q}_G(\textsf{A})\cong KR^{p-q}_G(\textsf{A})$. Second, in defining the Poincar\'e duality pairing between $KR$-homology and $KR$-theory, parallel to that between $K$-homology and $K$-theory, we shall need the notion of a Real $\text{Spin}^c$ structure as the $KR$-theoretic orientation. As we are unable to find a precise definition of Real $\text{Spin}^c$ structure in the literature, we shall hereby give one. 
	
	Let $V\to X$ be a $\Gamma$-equivariant, orientable real vector bundle of rank $n$ (caveat: this should be distinguished from Real vector bundles). Equip each fiber with an inner product such that $\sigma_V$ is orthogonal. Let $P$ be the oriented orthonormal frame bundle of $V$. $P$ has an $\mathit{SO}(n)$-action defined by 
	\[g\cdot(v_1, \cdots, v_{2n})=\left(\sum_{j=1}^{2n}g_{1, j}v_j, \cdots, \sum_{j=1}^{2n}g_{2n, j}v_j\right),\]
	making it a principal $\mathit{SO}(n)$-bundle. We make the structure group $\mathit{SO}(n)$ a Real Lie group by assigning the involutive automorphism
	\[g\mapsto \begin{pmatrix}I_q&\\ &-I_p\end{pmatrix}g\begin{pmatrix}I_q&\\ &-I_p\end{pmatrix}\]
	where $p+q=n$. We also provide $\text{Spin}^c(n)$ with the involutive automorphism which descends to the one on $\mathit{SO}(n)$ defined above and restricts to complex conjugation on the central circle. 
\begin{definition}\label{Realbundle}
		$P$ is a Real principal $\mathit{SO}(n)$-bundle of type $(p, q)$ if $P$ is a principal $SO(n)$-bundle equipped with the bundle involution
	\begin{eqnarray}\label{Realframe}\sigma_P(v_1, v_2, \cdots, v_{n})=(\sigma_V(v_1), \sigma_V(v_2), \cdots, \sigma_V(v_q), -\sigma_V(v_{q+1}), \cdots, -\sigma_V(v_{n})).\end{eqnarray}
\end{definition}
	Note that another way of saying Definition \ref{Realbundle} is that the (unoriented) frame bundle of $V$ with Real structure given by Equation (\ref{Realframe}) is the trivial double cover of $P$ in the Real sense\footnote{We cannot resist to point out that we thought of using the phrase `Really trivial double cover' but quickly realized that it was really a bad pun.}(a trivial double cover is trivial in the Real sense if the involution preserves the two connected components). Recall that an ordinary real vector bundle is orientable if its (unoriented) frame bundle is the trivial double cover of the oriented frame bundle. This prompts the following
	\begin{definition}
		Suppose $V$ is orientable. $V$ is \emph{Real }$(p, q)$-\emph{orientable} if the oriented frame bundle $P$ is preserved by $\sigma_P$ defined above.
	\end{definition}
	$V$ is Real $(p, q)$-orientable if and only if either $\sigma_V$ is orientation-preserving and $p$ is even, or $\sigma_V$ is orientation-reversing and $p$ is odd.
	
	\begin{example}\label{divbyfour}
		Consider the real vector bundle $\mathbb{R}^{r, s}\to \text{pt}$. It is Real $(p, q)$-orientable if and only if $q$ and $s$ are of the same parity or, equivalently, $p-q-(r-s)$ is divisible by 4.
	\end{example}
	\begin{definition}\label{realspinc}
		Let $V$ be a $\Gamma$-equivariant, orientable real vector bundle over a Real space $X$. Suppose further that $V$ is Real $(p, q)$-orientable. We say $V$ is \emph{Real} $(p, q)$-$\text{Spin}^c$ if there exists, over $X$, a Real principal $\text{Spin}^c(n)$-bundle $\widetilde{P}$ of type $(p, q)$ whose structure group lifts that of $P$ as the Real $U(1)$-central extension as defined above. Equivalently, $V$ is Real $(p, q)$-$\text{Spin}^c$ if $\bigwedge^*V\otimes_\mathbb{R}\mathbb{C}$ is of the form of the Real spinor bundle $S=\widetilde{P}\times_{\text{Spin}^c(n)}\mathbb{C}\text{l}_{p, q}$, where $\mathbb{C}\text{l}_{p, q}$ is the complex Clifford algebra $\mathbb{C}\text{l}_n$ equipped with an antilinear algebra involution $\sigma$ with $\sigma(\varepsilon_i)=\varepsilon_i$ for $1\leq i\leq q$ and $\sigma(\varepsilon_j)=-\varepsilon_j$ for $q+1\leq j\leq n$. The equivariant version of Real $(p, q)$-$\text{Spin}^c$ structure can be defined by incorporating the $G$-action compatibly throughout the above definition.
	\end{definition}
	With Definition \ref{realspinc}, we can define, similar to Example \ref{fundclass}, the fundamental class $[M]\in KR^G_{q, p}(M)$ of the Real $(p, q)$-$\text{Spin}^c$ Riemannian $G$-manifold $M$, and formulate the Poincar\'e duality for $KR$-theory: 
	\[KR_G^{r, s}(M)\stackrel{\cong}{\longrightarrow} KR_{q-r, p-s}^G(M).\]
	There is one technical point to note, however. When defining the Real Fredholm module $(\rho, L^2(M, S), D)$ of type $(p, q)$ representing $[M]$, the grading operators are given by the right multiplication by $\varepsilon_1, \cdots, \varepsilon_q, i\varepsilon_{q+1}, \cdots, i\varepsilon_{n}\in\mathbb{C}\text{l}_{p, q}$ on $S$. 
\begin{example}
	Consider again the $\Gamma$-equivariant real vector bundle $\mathbb{R}^{r, s}\to\text{pt}$, which we assume is Real $(p, q)$-orientable. The Real principal $\mathit{SO}(n)$-bundle $P$ of type $(p, q)$ is isomorphic, through the map $(v_1, \cdots, v_n)\mapsto \begin{pmatrix}v_1\\ \vdots\\ v_n\end{pmatrix}$, to $\mathit{SO}(n)\to\text{pt}$ with involution being (up to a rearrangement of the orthonormal vectors $v_1, \cdots, v_n$)
	\[g\mapsto \begin{pmatrix}-I_{|p-r|}&\\ & I_{n-|p-r|}\end{pmatrix}g.\]
Though obviously $\widetilde{P}=(\text{Spin}^c(n)=\text{Spin}(n)\times_{\mathbb{Z}_2}U(1)\to\text{pt})$ is the principal $\text{Spin}^c(n)$-bundle which lifts the structure group of $P$, making $\mathbb{R}^{n}\to\text{pt}$ a $\text{Spin}^c$ vector bundle, it is not true that one can always equip $\widetilde{P}$ with a compatible Real structure which descends to the one on $P$ so that $\mathbb{R}^{r, s}$ is Real $\text{Spin}^c$. We would like to determine when it is Real $(p, q)$-$\text{Spin}^c$. Note that Real $(p, q)$-orientability of $\mathbb{R}^{r, s}$ implies that $p-r$ is even. Let $|p-r|=2k$. The two elements in $\text{Spin}(n)$ that lift $\begin{pmatrix}-I_{2k}& \\ & I_{n-2k}\end{pmatrix}$ are $\pm\varepsilon_1\varepsilon_2\cdots\varepsilon_{2k}$. So $\mathbb{R}^{r, s}$ is Real $(p, q)$-$\text{Spin}^c$ if and only if the bundle automorphism on $\widetilde{P}$ given by $[(g, z)]\mapsto [(\pm\varepsilon_1\varepsilon_2\cdots\varepsilon_{2k}g, \overline{z})]$ is an involution, which is the case exactly if and only if $k$ is even. If $k$ is odd, the above bundle automorphism only acts as a `quarter turn' and may be deemed a `Quaternionic structure', as we will see in Section \ref{twistedkrhomology}.

\begin{proposition}\label{realspincreal}	Let $p+q=r+s=n$.
	\begin{enumerate}
		\item $\mathbb{R}^{r, s}\to\text{pt}$ is not Real $(p, q)$-orientable if and only if $\frac{p-q-(r-s)}{2}$ is odd.
		\item $\mathbb{R}^{r, s}\to\text{pt}$ is Real $(p, q)$-orientable but not Real $(p, q)$-$\text{Spin}^c$ if and only if $\frac{p-q-(r-s)}{4}$ is odd.
		 \item $\mathbb{R}^{r, s}\to\text{pt}$ is Real $(p, q)$-$\text{Spin}^c$ if and only if $p-q-(r-s)$ is divisible by 8.
	 \end{enumerate}
\end{proposition}
\end{example}	
\subsection{Equivariant Real Dixmier-Douady bundles and their classification}\label{eqrealDD}
The study of local coefficient systems for $K$-theory was pioneered in \cite{DK}. Since then various models for the local coefficient systems of $K$-theory and $K$-homology have been proposed, namely, bundle gerbes (cf. \cite{BCMMS}, \cite{CCM}, \cite{CW}, \cite{Mu}, \cite{MS}), projective Hilbert space bundles (\cite{AS2}) and Dixmier-Douady bundles, fiber bundles with fibers being the $C^*$-algebra $\mathcal{K}(\mathcal{H})$, the space of compact operators on a separable complex Hilbert space, and the structure group being $\text{Aut}(\mathcal{K}(\mathcal{H}))=PU(\mathcal{H})$ (cf. \cite{M1}, \cite{M2}, \cite{M3}, \cite{R}). In this paper we shall follow the convention in \cite{M2} to twist $K$-homology by Dixmier-Douady bundles (DD bundles for short). To adapt to our context of equivariant $KR$-homology, we shall define equivariant Real DD bundles and a Real version of third equivariant cohomology group which classifies equivariant Real DD bundles (up to Morita isomorphism, to be explained below).

\begin{definition}
	Let $X$ be a locally compact Real $G$-space. A $G$-\emph{equivariant Real DD bundle} $\mathcal{A}$ is a $G$-equivariant, locally trivial $\mathcal{K}(\mathcal{H})$-bundle with structure group $\mathit{PU}(\mathcal{H})$ with the compact open topology (cf. \cite[Appendix 1]{AS}), equipped with an involution $\sigma_\mathcal{A}$ such that 
	\begin{enumerate}
		\item $(\mathcal{A}, \sigma_\mathcal{A})$ is a Real $G$-space,
		\item $\sigma_\mathcal{A}$ descends to $\sigma_X$ on $X$, and 
		\item $\sigma_\mathcal{A}$ maps fiber to fiber anti-linearly. 
	\end{enumerate}
\end{definition}
\begin{definition}
	\begin{enumerate}
		\item An equivariant Real DD bundle $\mathcal{A}$ is \emph{Morita trivial} if there exists an equivariant Real Hilbert space bundle $\mathcal{E}$ such that $\mathcal{A}$ is isomorphic to $\mathcal{K}(\mathcal{E})$. We say $\mathcal{E}$ Morita trivializes $\mathcal{A}$ if $\mathcal{A}\cong\mathcal{K}(\mathcal{E})$. 
		\item The \emph{opposite DD bundle} of $\mathcal{A}$, denoted by $\mathcal{A}^{\text{opp}}$, is the $\mathcal{K}(\mathcal{H})$-bundle with the same underlying space as that of $\mathcal{A}$ except that it is modeled on the opposite Hilbert space $\mathcal{H}^\text{opp}$ with the conjugate complex structure. 
		\item The (completed) tensor product of two equivariant Real DD bundles $\mathcal{A}_1\otimes\mathcal{A}_2$ is the equivariant Real DD bundle modeled on $\mathcal{K}(\mathcal{H}_1)\otimes\mathcal{K}(\mathcal{H}_2)\cong\mathcal{K}(\mathcal{H}_1\otimes\mathcal{H}_2)$. 
		\item Two equivariant Real DD bundles $\mathcal{A}_1$ and $\mathcal{A}_2$ are \emph{Morita isomorphic} if $\mathcal{A}_1\otimes \mathcal{A}_2^{\text{opp}}$ is isomorphic to a Morita trivial equivariant Real DD bundle. We say $\mathcal{E}$ \emph{witnesses} the Morita isomorphism between $\mathcal{A}_1$ and $\mathcal{A}_2$ if $\mathcal{E}$ Morita trivializes $\mathcal{A}_1\otimes\mathcal{A}_2^{\text{opp}}$. 
		\item Suppose $(X_i, \mathcal{A}_i)$, $i=1, 2$ are two equivariant Real DD bundles modeled on $\mathcal{K}(\mathcal{H}_i)$. A \emph{Morita morphism}
		\[(f, \mathcal{E}): (X_1, \mathcal{A}_1)\to (X_2, \mathcal{A}_2)\]
		consists of an equivariant Real proper map $f: X_1\to X_2$ and an equivariant Real Hilbert space bundle $\mathcal{E}$ on $X_1$ which is a $(f^*\mathcal{A}_2, \mathcal{A}_1)$-bimodule, locally modeled on the $(\mathcal{K}(\mathcal{H}_2), \mathcal{K}(\mathcal{H}_1))$-bimodule $\mathcal{K}(\mathcal{H}_1, \mathcal{H}_2)$, the space of compact operators mapping from $\mathcal{H}_1$ to $\mathcal{H}_2$. 
	\end{enumerate}
\end{definition}	

If $(f, \mathcal{E}_i)$, $i=1, 2$, are two Morita morphisms from $(X_1, \mathcal{A}_1)$ to $(X_2, \mathcal{A}_2)$, then the $(f^*\mathcal{A}_2, \mathcal{A}_1)$-bimodules $\mathcal{E}_1$ and $\mathcal{E}_2$ differ by an equivariant Real line bundle $L$. More precisely, 
\[\mathcal{E}_2=\mathcal{E}_1\otimes L, \text{ with }L=\text{Hom}_{f^*\mathcal{A}_2-\mathcal{A}_1}(\mathcal{E}_1, \mathcal{E}_2).\]
\begin{definition}
	We say the two Morita morphisms $(f, \mathcal{E}_i)$, $i=1, 2$, are \emph{2-isomorphic} if the equivariant Real line bundle $L$ is trivial.
\end{definition}
It follows that, if there is a Morita morphism between $(X_1, \mathcal{A}_1)$ and $(X_2, \mathcal{A}_2)$ covering the equivariant Real map $f: X_1\to X_2$, the set of 2-isomorphism classes of Morita morphisms covering $f$ is a principal homogeneous space acted upon by the equivariant Real Picard group of $X_1$. 

The category of equivariant Real DD bundles over a Real $G$-space is therefore endowed with a grouplike monoidal structure where multiplication is given by the tensor product, and group inversion the opposite bundle construction. It is well-known that, analogous to complex line bundles being classified by the second integral cohomology group up to isomorphism, DD bundles are classified, up to Morita isomorphism, by the third integral cohomology group (cf. \cite{DiDo}). In what follows we will prove an analogous result for equivariant Real DD bundles by adding one additional structure (Real and equivariant structures) at a time to ordinary DD bundles. We first consider the classification of Real DD bundles. 

For a Real space $X$, let $\mathcal{U}=\{U_\alpha\}_{\alpha\in J}$ be a $\Gamma$-covering of $X$, i.e. a covering where the $\Gamma$-action on the index set $J$ is free and $\gamma\cdot U_\alpha=U_{\gamma\cdot\alpha}$. For a Real DD bundle $\mathcal{A}$, there are transition functions $r_{\alpha\beta}: U_{\alpha\beta}\to PU(\mathcal{H})$ satisfying $r_{\alpha\beta}(x)=\overline{r_{\gamma(\alpha)\gamma(\beta)}(\sigma_X(x))}$. If $\mathcal{U}$ is fine enough, $r_{\alpha\beta}$ can be lifted to a $U(\mathcal{H})$-valued function $\widehat{r}_{\alpha\beta}$. Let $s_{\alpha\beta\gamma}=\widehat{r}_{\alpha\beta}\widehat{r}_{\beta\gamma}\widehat{r}_{\gamma\alpha}$. $\{s_{\alpha\beta\gamma}\}$ defines a $\Gamma$-equivariant $U(1)$-valued 2-cocycle in the Cech cohomology group $H^2(\check{C}(\mathcal{U}, \underline{U(1)}_\Gamma)^\Gamma)$, where $\underline{U(1)}_\Gamma$ is the $\Gamma$-sheaf of continuous $U(1)$-valued functions with $\Gamma$ acting on $U(1)$ by complex conjugation. The short exact sequence 
\[1\longrightarrow\mathbb{Z}_\Gamma\longrightarrow\underline{\mathbb{R}}_\Gamma\longrightarrow\underline{U(1)}_\Gamma\longrightarrow 1\]
where $\Gamma$ acts on $\mathbb{Z}$ and $\mathbb{R}$ by negation, and the fact that $\underline{\mathbb{R}}_\Gamma$ is a fine $\Gamma$-sheaf (i.e. it admits a $\Gamma$-equivariant partition of unity), lead to the isomorphism $H^2(\check{C}(\mathcal{U}, \underline{U(1)}_\Gamma)^\Gamma)\cong H^3(\check{C}(\mathcal{U}, \mathbb{Z}_\Gamma)^\Gamma)$ induced by the coboundary map in the long exact sequence of Cech cohomology groups. By the Corollary in Section 5.5 of \cite{Gr}, $\displaystyle\lim_{\mathcal{U}}H^3(\check{C}(\mathcal{U}, \mathbb{Z}_\Gamma)^\Gamma)$ is isomorphic to the sheaf cohomology $H^3(X; \Gamma, \mathbb{Z}_\Gamma)$ defined as the third right derived functor of the $\Gamma$-invariant global section functor. According to \cite[Section 6]{St}, because $\Gamma$ is a discrete group, $H^3(X; \Gamma, \mathbb{Z}_\Gamma)$ is isomorphic to $H^3_\Gamma(X, \mathbb{Z}_\Gamma)$, which is Borel's equivariant cohomology defined more generally as follows.

\begin{definition}
	Let $G$ be a topological group with a $\Gamma$-action. Define the Real cohomology
	\[H_\Gamma^n(X, {G}_\Gamma):=H^n(X\times_\Gamma E\Gamma, \underline{X\times E\Gamma\times G/\Gamma})\]
	where $\underline{X\times E\Gamma\times G/\Gamma}$ is the local coefficient system with fiber $G$ over $X\times_\Gamma E\Gamma$. 
\end{definition}
\begin{definition}
	The Real DD-class of $\mathcal{A}$, denoted by $DD_\mathbb{R}(\mathcal{A})$, is defined to be the image of the 2-cocycle $\{s_{\alpha\beta\gamma}\}$ in $H^3_\Gamma(X, \mathbb{Z}_\Gamma)$ under the various isomorphisms of cohomology groups discussed above, namely $H^3(\check{C}(\mathcal{U}, \mathbb{Z}_\Gamma)^\Gamma)\cong H^3(X; \Gamma, \mathbb{Z}_\Gamma)\cong H^3_\Gamma(X, \mathbb{Z}_\Gamma)$. 
\end{definition}

\begin{remark}
	\begin{enumerate}
		\item The introduction of the Real cohomology $H_\Gamma^3(X, \mathbb{Z}_\Gamma)$ as the home where the Real DD-classes live is inspired by \cite{Kah}, where Kahn introduced Real Chern classes and use $H^2_\Gamma(X, \mathbb{Z}_\Gamma)$ to classify Real line bundles over $X$. 
		\item If $(f, \mathcal{E}): (X_1, \mathcal{A}_1)\to (X_2, \mathcal{A}_2)$ is a Morita morphism, then $f^*\text{DD}_\mathbb{R}(\mathcal{A}_2)=\text{DD}_\mathbb{R}(\mathcal{A}_1)$. The map $\text{DD}_\mathbb{R}: \text{the group of Morita isomorphism classes of Real DD bundles over }X\to H^3_\Gamma(X, \mathbb{Z}_\Gamma)$ is a group homomorphism. 
	\end{enumerate}
\end{remark}

By definition, one can see that $\mathcal{A}$ is Morita trivial if and only if $\text{DD}_\mathbb{R}(\mathcal{A})=0\in H_\Gamma^3(X, \mathbb{Z}_\Gamma)$. In fact, we have
\begin{proposition}
	The Morita isomorphism classes of Real DD bundles over $X$ is classified by $H_\Gamma^3(X, \mathbb{Z}_\Gamma)$.
\end{proposition}
\begin{proof}
	It remains to show that for any class $\alpha\in H_\Gamma^3(X, \mathbb{Z}_\Gamma)$, there exists a Real DD bundle $\mathcal{A}$ such that $DD_\mathbb{R}(\mathcal{A})=\alpha$. By a result in \cite{S3}, an incarnation for $EU(1)$ is the space of step functions taking values in $U(1)$ and constant on the intervals $(t_i, t_{i+1})$ for some partition $0=t_0<\cdots<t_{m+1}=1$ for the unit interval. This model can be made an abelian group, where the group law is simply given by pointwise multiplication. The group $U(1)$ can be regarded as the normal subgroup of $EU(1)$, which is the subspace of $U(1)$-valued functions constant on the unit interval. Hence we have the short exact sequence
	\[1\longrightarrow U(1)\longrightarrow EU(1)\longrightarrow BU(1)\longrightarrow 1.\]
	Note that there is a $\Gamma$-action on $EU(1)$ induced by the conjugation action on $U(1)$, and $BU(1)$ inherits this $\Gamma$-action as well. One may iterate Segal's construction to the abelian group $BU(1)$, and get the short exact sequence
	\[1\longrightarrow BU(1)\longrightarrow EBU(1)\longrightarrow BBU(1)\longrightarrow 1.\]
	Again the three groups in the above short exact sequence possess naturally induced $\Gamma$-action. The coboundary maps of the long exact sequences of cohomology groups induced by the the following short exact sequences of $\Gamma$-sheaves
	\begin{center}
		$0\longrightarrow\mathbb{Z}_\Gamma\longrightarrow\underline{\mathbb{R}}_\Gamma\longrightarrow\underline{U(1)}_\Gamma\longrightarrow 0$\\
		$0\longrightarrow\underline{U(1)}_\Gamma\longrightarrow\underline{EU(1)}_\Gamma\longrightarrow\underline{BU(1)}_\Gamma\longrightarrow 0$\\
		$0\longrightarrow\underline{BU(1)}_\Gamma\longrightarrow\underline{EBU(1)}_\Gamma\longrightarrow\underline{BBU(1)}_\Gamma\longrightarrow 0$\\
	\end{center}
	give rise to the following string of isomorphisms
	\begin{align*}
		H^3_\Gamma(X, \mathbb{Z}_\Gamma)
		\cong& H^2_\Gamma(X, \underline{U(1)}_\Gamma)\\
		\cong& H^1_\Gamma (X, \underline{BU(1)}_\Gamma)\\
		\cong& H^0_\Gamma (X, \underline{BBU(1)}_\Gamma)\\
		\cong& [X, BBU(1)]_\mathbb{R}
	\end{align*}
	where $[X, Y]_\mathbb{R}$ is the set of $\Gamma$-equivariant homotopy equivalence classes of Real maps from $X$ to $Y$. So $BBU(1)$ is a representing space for $H_\Gamma^3(\cdot, \mathbb{Z}_\Gamma)$. In fact, the space $BPU(\mathcal{H})$ (equipped with the $\Gamma$-action induced by conjugation on $\mathcal{H}$) can be used as a representing space in place of $BBU(1)$, because they are $\Gamma$-equivariant homotopy equivalent, which is shown below.
	
	Let $\mathcal{H}$ be the space of $L^2$-integrable complex-valued functions on the unit interval. We may define a map 
	\begin{align*}
		F: EU(1)&\to U(\mathcal{H})\\
		f&\mapsto m_f
	\end{align*}
	where $m_f$ is the multiplication by $f$ operator. $F$ is both $U(1)$- and $\Gamma$-equivariant, and a homotopy equivalence (both $EU(1)$ and $U(\mathcal{H})$ are contractible). So $F$ induces a $\Gamma$-equivariant homotopy equivalence between $BU(1)$ and $PU(\mathcal{H})$, and hence another $\Gamma$-equivariant homotopy equivalence between $BBU(1)$ and $BPU(\mathcal{H})$. 
	
	We have that, given $\alpha\in H_\Gamma^3(X, \mathbb{Z}_\Gamma)$, there exists a Real map $f: X\to BPU(\mathcal{H})$ whose $\Gamma$-equivariant homotopy class corresponds to $\alpha$. The Real DD bundle $f^*(EPU(\mathcal{H})\times_{PU(\mathcal{H})}\mathcal{K}(\mathcal{H}))$ can be easily checked to have $\alpha$ as the Real DD-class.
\end{proof}

\begin{remark}
	The proof above also shows that the Real DD bundle $EPU(\mathcal{H})\times_{PU(\mathcal{H})}\mathcal{K}(\mathcal{H})$ is a universal Real DD bundle.
\end{remark}

Let $G\rtimes \Gamma$ be the semi-direct product such that conjugation by $\gamma$ on $G$ amounts to the automorphic involution $\sigma_G$. One can define the equivariant version of the Real cohomology similarly by Borel's construction. 
\begin{definition}
	Let $X$ be a Real $G$-space. Define the equivariant Real cohomology $H^n_{G\rtimes\Gamma}(X, \mathbb{Z}_\Gamma)$ to be $H^n((X\times_G EG)\times_\Gamma E\Gamma, \underline{(X\times_G EG)\times E\Gamma\times\mathbb{Z}/\Gamma})$. Here $\gamma$ acts on $\mathbb{Z}$ by negation.
\end{definition}
Atiyah and Segal showed in \cite{AS2} that the equivariant cohomology $H_G^3(X, \mathbb{Z})$ classifies the equivariant DD bundles up to Morita isomorphism. Adapting their arguments to the Real context, one can show that 
\begin{proposition}
	The Morita isomorphism classes of equivariant Real DD bundles on the Real $G$-space $X$ is classified by $H_{G\rtimes\Gamma}^3(X, \mathbb{Z}_\Gamma)$.
\end{proposition}
The following simple observation is helpful in computing the free part of $H^*_{G\rtimes\Gamma}(X, \mathbb{Z}_\Gamma)$.

\begin{proposition}\label{antiinvcoh}
	Let $X$ be a smooth Real $G$-manifold. $H_{G\rtimes\Gamma}^n(X, \mathbb{R}_\Gamma)$ is isomorphic to $H^n_G(X, \mathbb{R})^{-\Gamma}$, the $\Gamma$-anti-invariant subgroup of $H^n_G(X, \mathbb{R})$ ($\alpha\in H^n_G(X, \mathbb{R})=H^n(X\times_G EG, \mathbb{R})$ is said to be $\Gamma$-anti-invariant if $\gamma^*\alpha=-\alpha$, where $\gamma^*$ is the pullback map induced by the involutions on $X$ and $G$). 
\end{proposition}
\begin{proof}
	Viewing the space $(X\times_G EG)\times_\Gamma E\Gamma$ as a fiber bundle over $B\Gamma$ with $X\times_G EG$ as the fiber, we have that the $E_2$-page of the Serre spectral sequence for the cohomology group $H^n((X\times_G EG)\times_\Gamma E\Gamma, \underline{(X\times_G EG)\times_\Gamma E\Gamma\times\mathbb{R}/\Gamma})$ is 
	\[E_2^{p, q}=H^p(B\Gamma, \underline{E\Gamma\times H^q(X\times_G EG, \mathbb{R})/\Gamma})\]
	where $\gamma$ acts on $H^q(X\times_G EG, \mathbb{R})$ by $-\gamma^*$ (the minus sign comes from the negation action on $\mathbb{R}$ by $\gamma$). From the Gysin sequence (\ref{gysinseq}) in Example \ref{twotorsiontwist} below with the integral coefficient $\mathbb{Z}$ replaced by the real coefficient $\mathbb{R}$, and the vanishing of $H^0(B\Gamma, \underline{E\Gamma\times_\Gamma\mathbb{R}})$ because of the non-existence of nonzero global section of local coefficient system $\underline{E\Gamma\times_\Gamma\mathbb{R}}$, we observe that $H^p(B\Gamma, \underline{E\Gamma\times_\Gamma\mathbb{R}})=0$ for $p\geq 0$\footnote{That $H^p(B\Gamma, \underline{E\Gamma\times_\Gamma\mathbb{R}})=0$ for $p>0$ can be seen as follows. Consider the aforementioned Gysin sequence
	\begin{eqnarray*}
	\label{gysinseq}\cdots\longrightarrow H^{p-1}(E\Gamma, \mathbb{R})\stackrel{\int}{\longrightarrow}H^{p-1}(B\Gamma, \mathbb{R})\stackrel{\cup e}{\longrightarrow} H^p(B\Gamma, \underline{E\Gamma\times_\Gamma\mathbb{R}})\stackrel{\pi^*}{\longrightarrow} H^p(E\Gamma, \mathbb{R})\longrightarrow\cdots
	\end{eqnarray*}
	and the following two cases.
	\begin{enumerate}
		\item When $p=1$, the map $\int$ is an isomorphism and $H^1(E\Gamma, \mathbb{R})=0$ as $E\Gamma$ is contractible, so $H^1(B\Gamma, \underline{E\Gamma\times_\Gamma\mathbb{R}})=0$.
		\item When $p>1$, both $H^{p-1}(B\Gamma, \mathbb{R})$ and $H^p(E\Gamma, \mathbb{R})$ are 0, so $H^p(B\Gamma, \underline{E\Gamma\times_\Gamma\mathbb{R}})=0$ as well. 
	\end{enumerate}}. This, together with the fact that $H^*(B\Gamma, \mathbb{R})\cong\mathbb{R}$, implies that $E_2^{p, q}$ vanishes for $p>0$ and is $H^0(B\Gamma, \underline{E\Gamma\times H^q(X\times_G EG, \mathbb{R})/\Gamma})=H^q(X\times_G EG, \mathbb{R})^{-\Gamma}$ for $p=0$, and that the spectral sequence collapses on the $E_2$-page. This completes the proof.
	%$H^n_{G\rtimes\Gamma}(X, \mathbb{R}_\Gamma)$ is isomorphic to the sheaf cohomology $H^n(X\times_G EG; \Gamma, \mathbb{R}_\Gamma)$ (cf. \cite[Section 6]{St}). $\mathbb{R}_\Gamma$ has the following resolution by fine $\Gamma$-sheaves
	%\[\mathbb{R}_\Gamma\stackrel{d}{\longrightarrow}\underline{\Omega^0(X\times_G EG)}\stackrel{d}{\longrightarrow}\underline{\Omega^1(X\times_G EG)}\stackrel{d}{\longrightarrow}\cdots\]
	%Note that $\Gamma$ acts on the global sections of $\underline{\Omega^n(X\times_G EG)}$ by $-\gamma^*$. We have that $H^n(X; \Gamma, \mathbb{R}_\Gamma)$ is the $n$-th cohomology of the cochain
	%\[\Omega^0(X\times_G EG)^{-\Gamma} \stackrel{d}{\longrightarrow}\Omega^1(X\times_G EG)^{-\Gamma} \stackrel{d}{\longrightarrow}\cdots\]
	%An easy averaging argument shows that taking $\Gamma$-anti-invariants and taking cohomology commute. Hence $H^n(X\times_G EG; \Gamma, \mathbb{R}_\Gamma)\cong H^n(X\times_G EG, \mathbb{R})^{-\Gamma}\cong H_G^*(X, \mathbb{R})^{-\Gamma}$. 
\end{proof}

\begin{example}\label{twotorsiontwist}
	We shall show that $H_\Gamma^3(\text{pt}, \mathbb{Z}_\Gamma)\cong\mathbb{Z}_2$ and exhibit the DD bundle whose DD class is the nontrivial 2-torsion. Consider the $S^0$-bundle $E\Gamma\to B\Gamma$. Being non-orientable, this bundle has Euler class $e$ which lives in the first cohomology group $H^1(B\Gamma, \underline{E\Gamma\times_\Gamma\mathbb{Z}})$ with twisted coefficient system. Applying the Gysin sequence
	\begin{eqnarray}
	\label{gysinseq}\cdots\longrightarrow H^2(E\Gamma, \mathbb{Z})\stackrel{\int}{\longrightarrow}H^2(B\Gamma, \mathbb{Z})\stackrel{\cup e}{\longrightarrow} H^3(B\Gamma, \underline{E\Gamma\times_\Gamma\mathbb{Z}})\stackrel{\pi^*}{\longrightarrow} H^3(E\Gamma, \mathbb{Z})\longrightarrow\cdots
	\end{eqnarray}
	and by its exactness we have that $H^3(B\Gamma, \underline{E\Gamma\times_\Gamma\mathbb{Z}})\cong H^2(B\Gamma, \mathbb{Z})\cong\mathbb{Z}_2$. Thus, there is at least a copy of $\mathbb{Z}_2$ as a summand in $H^3_{G\rtimes\Gamma}(X, \mathbb{Z}_\Gamma)$ which is the image of the split-injective pullback map $H^3_\Gamma(\text{pt}, \mathbb{Z}_\Gamma)\to H^3_{G\rtimes\Gamma}(X, \mathbb{Z}_\Gamma)$. The nonzero 2-torsion of $H^3_\Gamma(\text{pt}, \mathbb{Z}_\Gamma)$ is the Real DD class of the $\mathcal{K}(\mathcal{H})$-bundle over a point where $\mathcal{H}=l^2(\mathbb{N})$ and the involution on $\mathcal{K}(\mathcal{H})$ is induced by the `quaternionic quarter turn' on $\mathcal{H}$, i.e. $(z_1, z_2, z_3, z_4, \cdots)\mapsto (-\overline{z_2}, \overline{z_1}, -\overline{z_4}, \overline{z_3}, \cdots)$. In Section \ref{twistedkrhomology} we give another interpretation of this non-zero 2-torsion class as the obstruction for the $\Gamma$-equivariant real vector bundle $\mathbb{R}^{r, s}\to \text{pt}$ to possess a Real $(p, q)$-$\text{Spin}^c$ structure or, equivalently, the $KR$-theory orientation, when $\frac{p-q-(r-s)}{4}$ is odd. 
\end{example}

\begin{example}\label{FrobSchur}
	Let $X$ be a $G$-space with trivial Real structure. The equivariant Real cohomology group $H^3_{G\times \Gamma}(X, \mathbb{Z}_\Gamma)$ is, by K\"unneth formula, isomorphic to 
	\[\bigoplus_{i=0}^3 H_G^i(X, \mathbb{Z})\otimes H^{3-i}(B\Gamma, \underline{E\Gamma\times_\Gamma\mathbb{Z}})\oplus \bigoplus_{i=0}^4\text{Tor}_1(H_G^i(X, \mathbb{Z}), H^{4-i}(B\Gamma, \underline{E\Gamma\times_\Gamma\mathbb{Z}})).\]
	The Gysin sequence (\ref{gysinseq}) shows that $H^k(B\Gamma, E\Gamma\times_\Gamma\mathbb{Z})$ is 0 when $k$ is even, and $\mathbb{Z}_2$ when $k$ is odd. Moreover, by the Universal Coefficient Theorem, $H_G^1(X, \mathbb{Z})\cong \text{Hom}(H_1(X\times_G EG, \mathbb{Z}), \mathbb{Z})$, which is a free $\mathbb{Z}$-module. So the term $\text{Tor}_1(H_G^1(X, \mathbb{Z}), \mathbb{Z}_2)$ vanishes, and we can rewrite the isomorphism as
	\[H_{G\times\Gamma}^3(X, \mathbb{Z}_\Gamma)\cong\pi_0(X)\otimes\mathbb{Z}_2\oplus H_G^2(X, \mathbb{Z})\otimes\mathbb{Z}_2\oplus \text{Tor}_1(H_G^3(X, \mathbb{Z}), \mathbb{Z}_2).\]
	By another version of the Universal Coefficient Theorem, $H_G^2(X, \mathbb{Z})\otimes\mathbb{Z}_2\oplus\text{Tor}_1(H_G^3(X, \mathbb{Z}), \mathbb{Z}_2)$ is isomorphic to $H_G^2(X, \mathbb{Z}_2)$. We can finally put the above isomorphism as 
	\[H_{G\times\Gamma}^3(X, \mathbb{Z}_\Gamma)\cong \pi_0(X)\otimes\mathbb{Z}_2\oplus H_G^2(X, \mathbb{Z}_2).\]
	The cohomology group $H_G^2(X, \mathbb{Z}_2)$ is known to classify equivariant real DD-bundles, which twist equivariant $KO$-theory. The above isomorphism shows that the group $H_{G\times\Gamma}^3(X, \mathbb{Z}_\Gamma)$ classifies equivariant real twists, as is expected from the assumption of the triviality of the Real structure, and more. The extra piece $\pi_0(X)\otimes\mathbb{Z}_2$ serves as the `Frobenius-Schur indicator' which specifies on each connected component of $X$ the type of involution on the equivariant Real DD-bundle, where the real type comes from the complex conjugation on the complex Hilbert space $\mathcal{H}$ and the quaternionic type comes from the `quaternionic quarter turn' on $\mathcal{H}$ as in Example \ref{twotorsiontwist}. 
\end{example}

\begin{remark}\label{diffmo}
	Using a \v{C}ech construction, Moutuou in \cite{Mo2} defined an equivariant cohomology group $HR^2(X, \mathcal{T})$, where $\mathcal{T}$ is the sheaf of continuous $S^1$-value functions on $X$ with the $\Gamma$-action induced by the involution on $X$ and complex conjugation on $S^1$, and used it to classify the Real twists of groupoids, which in turn is used to construct twisted $KR$-theory of groupoids in \cite{Mo}. As noted in \cite[Theorem 2.5]{R2}, Moutuou's $HR^*$ group is not the same as Grothendieck's equivariant sheaf cohomology group or equivalently Borel's equivariant cohomology group we use in this paper. For instance, $HR^*(\text{pt}, \mathcal{T})\cong H^*(\text{pt}, \mathbb{Z}_2)\cong\mathbb{Z}_2$ whereas $H^k_\Gamma(\text{pt}, \mathbb{Z}_\Gamma)\cong\mathbb{Z}_2$ if $k$ is odd. Borel's equivariant cohomology group $H_\Gamma^3(\text{pt}, \mathbb{Z}_\Gamma)$ can detect the nontrivial Real twist over a point (i.e. the Real DD bundle exhibited in Example \ref{twotorsiontwist}), while $HR^2(\text{pt}, \mathcal{T})$ cannot.
\end{remark}

\subsection{Twisted $K$-theory (homology)}\label{twistedkrhomology}

Twisted $K$-theory, or $K$-theory with local coefficient systems, was first studied in \cite{DK}, where the case of local coefficient systems with torsion DD-classes was explored. The general case was taken up in \cite{R}, where Rosenberg defined twisted $K$-theory as homotopy equivalence classes of sections of a twisted bundle of Fredholm operators, with twisting data given by the local coefficient system. In recent years there have been extensive works done on twisted $K$-theory and its various models (cf. \cite{AS2}, \cite{BCMMS}, \cite{CW}, \cite{Kar} and the references therein) due to its connection with string theory. It has been conjectured that $D$-branes and Ramond-Ramond field strength are classified by twisted $K$-theory, where the twist is defined by a $B$-field. An impetus to the whole enterprise of studying this mathematical physical connection is a deep result by the Freed-Hopkins-Teleman, which is explained in more details in Section \ref{realFHT}.  

In what follows, we shall use the following definition of twisted $KR$-theory (homology) obtained by incorporating the Real structures into the analytic definition of twisted $K$-theory and homology, as in \cite{M1}.

\begin{definition}
	For a Real $G$-space $X$ with an equivariant Real DD bundle $\mathcal{A}$,  we define the \emph{twisted $KR$-homology}
	\[KR_q^G(X, \mathcal{A}):=KR_G^{-q}(S_0(\mathcal{A}))\]
	where $S_0(\mathcal{A})$ is the Real $G-C^*$-algebra of space of sections of $\mathcal{A}$ vanishing at infinity. Similarly, we define the twisted $KR$-theory
	\[KR_G^q(X, \mathcal{A}):=KR^G_{-q}(S_0(\mathcal{A})).\]
\end{definition} 
We list some useful features of twisted equivariant (Real) $K$-homology, adapted from \cite{M1}. 
\begin{enumerate}
	\item\label{canonicalmorita} A Morita morphism $(f, \mathcal{E}): (X_1, \mathcal{A}_1)\to (X_2, \mathcal{A}_2)$ induces a pushforward map 
	\[f_*: KR^G_*(X_1, \mathcal{A}_1)\to KR^G_*(X_2, \mathcal{A}_2)\]
	which depends only on the 2-isomorphism class of $(f, \mathcal{E})$. In particular, the push forward map $(\text{Id}, (X\times\mathcal{H})\otimes L)_*$ induces an automorphism on $KR^G_*(X, \mathcal{A})$, which depends on the isomorphism class of the equivariant Real line bundle $L$. In other words, the equivariant Real Picard group of $X$ maps to the automorphism group of $KR^G_*(X, \mathcal{A})$. If the equivariant Real Picard group of $X_1$, which is $H^2_{G\rtimes\Gamma}(X_1, \mathbb{Z}_\Gamma)$ by straightforwardly adapting the arguments in Section \ref{eqrealDD}, is trivial, then there is only one canonical pushforward map independent of $\mathcal{E}$. From now on, we will, for brevity, sometimes use $f_*$ to denote the pushforward map if it is independent of $\mathcal{E}$. 
	\item \label{crossedprod}One can define the cross product, which is a special case of the Kasparov product (cf. \cite{HR})
	\[KR_*^G(X_1, \mathcal{A}_1)\otimes KR_*^G(X_2, \mathcal{A}_2)\to KR_*^G(X_1\times X_2, \mathcal{A}_1\otimes\mathcal{A}_2)\]
	where $\mathcal{A}_1\otimes \mathcal{A}_2$ is the external tensor product. 
	\item Recall that one of the motivations for introducing local coefficient systems to generalized cohomology theory is to formulate Thom isomorphism and Poincar\'e duality for spaces which are non-orientable in the sense of the relevant cohomology theory. For an even rank (resp. odd rank) $G$-equivariant real vector bundle $V\to X$, which is orientable in the usual sense but not necessarily $G$-$\text{Spin}^c$ (i.e. equivariant $K$-theoretic orientable), the $K$-theoretic local coefficient system reflecting the obstruction for $V$ to be $G$-$\text{Spin}^c$ can be realized by the Clifford bundle $\mathbb{C}\text{l}(V)$ (resp. $\mathbb{C}\text{l}(V\oplus\underline{\mathbb{R}})$), which is a matrix algebra bundle (and hence a DD bundle) with its DD-class being 2-torsion. This is because $V$ is $G$-$\text{Spin}^c$ if and only if its Clifford bundle is Morita trivial (and trivialized by the reduced spinor bundle). We will call $\mathbb{C}\text{l}(V)$ (resp. $\mathbb{C}\text{l}(V\oplus\underline{\mathbb{R}})$) and any other Morita isomorphic DD bundles the orientation twist of $V$ and denote it by $o_V$. The Thom isomorphism now can be formulated as
	\[K_G^*(X, \mathcal{A}\otimes o_V)\cong K_G^{*+n}(V, \pi^*\mathcal{A})\]
	(cf. \cite[Theorems 3.5, 3.9]{CW} for the nonequivariant case and \cite[Section 5.6]{Kar} for the equivariant case) whereas the Poincar\'e duality is
	\[K_{n-*}^G(X, \mathcal{A})\cong K_G^{*}(X, \mathcal{A}^{\text{opp}}\otimes o_{TX})\]
	(cf. \cite[Theorem 2.1]{Tu}). The corresponding statement of Thom isomorphism and Poincar\'e duality in the Real case are completely analogous, just that we need to treat the degree shift carefully. Let $V$ and $TX$ be Real $(p, q)$-orientable Euclidean $G$-vector bundles. Then the Thom isomorphism should be
	\[KR_G^{r+p, s+q}(X, \mathcal{A}\otimes o_V)\cong KR_G^{r, s}(V, \pi^*\mathcal{A})\]
	while the Poincar\'e duality is
	\[KR^{r, s}_G(X, \mathcal{A})\cong KR^G_{q-r, p-s}(X, \mathcal{A}^{\text{opp}}\otimes o_{TX}).\]
\end{enumerate}	
\begin{example}
	By Proposition \ref{realspincreal}, $\mathbb{R}^{r, s}\to\text{pt}$ is Real $(p, q)$-$\text{Spin}^c$ if $p-q-(r-s)$ is divisible by 8. The Thom isomorphism together with $(1, 1)$-periodicity gives
	\[KR^{i+p-j-q}_G(X)\cong KR_G^{i+p, j+q}(X)\cong KR^{i, j}_G(X\times\mathbb{R}^{r, s})\cong KR_G^{i+r, j+s}(X)\cong KR_G^{i+r-j-s}(X)\]
	which is consistent with the 8-periodicity of $KR$-theory. On the other hand, if $p-q-(r-s)$ is only divisible by 4 but not 8, then $\mathbb{R}^{r, s}\to\text{pt}$ is only Real $(p, q)$-orientable but not Real $(p, q)$-$\text{Spin}^c$. The DD class of $o_{\mathbb{R}^{r, s}}$ is exactly the non-zero two torsion of the image of the split-injective map $H^3_\Gamma(\text{pt}, \mathbb{Z}_\Gamma)\to H_{G\rtimes\Gamma}^3(\text{pt}, \mathbb{Z}_\Gamma)$. In particular, if $r=0$, $s=4$, $p=q=2$, the twisted version of Thom isomorphism yields 
	\begin{proposition}\label{twistedrealquat}
		$KR_G^*(X, \pi^*(o_{\mathbb{R}^{0, 4}}))\cong KR_G^*(X\times\mathbb{R}^{0, 4})$, which in turn is isomorphic to the Quaternionic $K$-theory $KH_G^*(X)$ by Proposition \ref{shiftbyfour}. 
	\end{proposition}
	
	So Quaternionic $K$-theory is Real $K$-theory twisted by $o_{\mathbb{R}^{0, 4}}$. 
\end{example}

\section{The equivariant fundamental DD bundle over $G^-$}\label{eqfundDD}
\subsection{The group of Morita isomorphism classes of equivariant Real DD bundles over $G^-$}\label{MorG}
In this section, we shall compute $H^3_{G\rtimes\Gamma}(G^-, \mathbb{Z}_\Gamma)$, the group of Morita isomorphism classes of $G$-equivariant Real DD-bundles over $G^-$. It is well-known that $H_G^3(G, \mathbb{R})\cong\mathbb{R}$ and the integral generator is represented by the equivariant differential form
\[\eta_G(\xi)=\frac{1}{12}B(\theta^L, [\theta^L, \theta^L])-\frac{1}{2}B(\theta^L+\theta^R, \xi)\]
where $\theta^L$ and $\theta^R$ are the left and right Maurer-Cartan forms respectively, and $\xi\in\mathfrak{g}$ (cf. \cite[Section 3]{M2}). 

\begin{lemma}\label{invgamma}
	$H^3_G(G, \mathbb{R})$ is $\Gamma$-anti-invariant. 
\end{lemma}
\begin{proof}
	It suffices to show that $\eta_G$ is $\Gamma$-anti-invariant by Proposition \ref{antiinvcoh}. Note that
\begin{align*}
	-(\gamma^*\eta_G)(\xi)&=-(a_G)^*(\eta_G(\sigma_G(\xi)))\\
						&=-\frac{1}{12}B((a_G)^*\theta^L, [(a_G)^*\theta^L, (a_G)^*\theta^L])+\frac{1}{2}B((a_G)^*(\theta^L+\theta^R), \sigma_G(\xi))\\
	                                         &=-\frac{1}{12}B(-\sigma_G^*\theta^R, \sigma_G^*[\theta^R, \theta^R])+\frac{1}{2}B(-\sigma_G^*(\theta^L+\theta^R), \sigma_G(\xi))\\
	                                         &=\frac{1}{12}B(\theta^R, [\theta^R, \theta^R])-\frac{1}{2}B(\theta^L+\theta^R, \xi)\\
	                                         &=\eta_G(\xi)
\end{align*}
The result follows.
\end{proof}
By Lemma \ref{antiinvcoh}, $H^3_{G\rtimes\Gamma}(G^-, \mathbb{R}_\Gamma)\cong H^3_G(G^-, \mathbb{R})^{-\Gamma}\cong\mathbb{R}$. It follows that $H^3_{G\rtimes\Gamma}(G^-, \mathbb{Z}_\Gamma)$ is of rank 1 and its free part is generated by $[\eta_G]$. On the other hand, there must be a $\mathbb{Z}_2$ summand in $H^3_{G\rtimes\Gamma}(G^-, \mathbb{Z}_\Gamma)$ by the discussion in Example \ref{twotorsiontwist}. The non-zero 2-torsion is $\text{DD}_\mathbb{R}(\pi^*o_{\mathbb{R}^{0, 4}})$. Thus we have that $H^3_{G\rtimes\Gamma}(G^-, \mathbb{Z}_\Gamma)$ contains the subgroup generated by $[\eta_G]$ and $\text{DD}_\mathbb{R}(\pi^*o_{\mathbb{R}^{0, 4}})$. In fact, it is all that this equivariant Real cohomology group contains.
\begin{proposition}\label{RealcohLie}
	$H^3_{G\rtimes\Gamma}(G^-, \mathbb{Z}_\Gamma)\cong \mathbb{Z}[\eta_G]\oplus \mathbb{Z}_2\text{DD}_\mathbb{R}(\pi^*o_{\mathbb{R}^{0, 4}})$. 
\end{proposition}
\begin{proof}
	By definition, 
	\[H_{G\rtimes\Gamma}^3(G^-, \mathbb{Z}_\Gamma)=H^3(((G^-\times EG)/G\times E\Gamma)/\Gamma, \underline{((G^-\times EG)/G\times E\Gamma\times\mathbb{Z})/\Gamma}).\]
	Applying Serre spectral sequence to the fiber bundle $(G^-\times EG)/G\hookrightarrow ((G^-\times EG)/G\times E\Gamma)/\Gamma\to B\Gamma$, we have that the $E_2$-page is 
	\[E_2^{p, q}=H^p(B\Gamma, H^q_G(G^-, \mathbb{Z})\times_\Gamma E\Gamma)\]
	Note that
	\begin{align*}
		E_2^{0, 3}&=H^0(B\Gamma, \mathbb{Z})\cong\mathbb{Z}\ \ (H^3_G(G, \mathbb{Z})\cong\mathbb{Z}\text{ is invariant under the $\Gamma$-action by Lemma \ref{invgamma}})\\
		E_2^{1, 2}&=E_2^{2, 1}=0\text{ as }H_G^i(G, \mathbb{Z})=0\text{ for }i=1, 2\\
		E_2^{3, 0}&=H_\Gamma^3(\text{pt}, \mathbb{Z}_\Gamma)\cong\mathbb{Z}_2\ \ (\text{Note that }\Gamma\text{ acts on }H_G^0(G^-, \mathbb{Z})\cong\mathbb{Z}\text{ by negation})
	\end{align*}
	The convergence of the spectral sequence implies that $H_{G\rtimes\Gamma}^3(G^-, \mathbb{Z}_\Gamma)$ is a certain extension of subquotients of $\mathbb{Z}_2$ by $\mathbb{Z}$. But from the discussion preceding this Proposition we have that $H^3_{G\rtimes\Gamma}(G^-, \mathbb{Z}_\Gamma)$ contains $\mathbb{Z}\oplus \mathbb{Z}_2$ as a subgroup. We conclude that indeed $H_{G\rtimes\Gamma}^3(G^-, \mathbb{Z}_\Gamma)\cong\mathbb{Z}\oplus\mathbb{Z}_2$. 
\end{proof}
\begin{definition}\label{eqrealfundDD}
	Any equivariant Real DD bundles over $G^-$ whose equivariant Real DD class is $[-\eta_G]$ is called an equivariant Real fundamental DD bundle.
\end{definition}

\begin{remark}\label{infiniteorderDD}
		Using the ideas in this Section, one can show that $H^3_{G\rtimes\Gamma}(G, \mathbb{Z}_\Gamma)\cong\mathbb{Z}_2\text{DD}_\mathbb{R}(\pi^*o_{\mathbb{R}^{0, 4}})$ if $\Gamma$ acts on $G$ by an involutive automorphism. In this case the equivariant twisted $KR$-homology of $G$, $KR^G_*(G, \pi^*(o_{\mathbb{R}^{0, 4}}))$ is isomorphic to $KR^G_{*-4}(G)\cong KH^G_*(G)$ by Proposition \ref{twistedrealquat}. We refer the reader to \cite{F} for a complete description of the algebra structure of $KR^*_G(G)$ (and hence $KR^G_*(G)$ by Poincar\'e duality and $KH^G_*(G)$ by shifting degree by $-4$).
\end{remark}

\subsection{A distinguished maximal torus with respect to $\sigma_G$}\label{distingtor} 
	In contructing the equivariant Real fundamental DD-bundle over $G$, which is done in Section \ref{equivrealfundDD}, we need to work with a particular kind of maximal torus associated with $\sigma_G$. In this Section we record the results about this maximal torus, directly taken from \cite{OS}. 
	
Let $\mathfrak{g}=\mathfrak{k}\oplus\mathfrak{q}$, where $\mathfrak{k}$ and $\mathfrak{q}$ are the $\pm 1$ eigenspaces of $\sigma_{\mathfrak{g}}$ respectively. Let $\mathfrak{a}$ be the maximal abelian subspace of $\mathfrak{q}$, $\mathfrak{t}$ a choice of maximal abelian subalgebra of $\mathfrak{g}$ containing $\mathfrak{a}$. Let $\mathfrak{k}'$ be the centralizer of $\mathfrak{a}$ in $\mathfrak{k}$, and $\mathfrak{t}'=\mathfrak{t}\cap\mathfrak{k}'$. It is known that $\mathfrak{t}=\mathfrak{t}'\oplus\mathfrak{a}$ (c.f. \cite{OS}, Appendix B) and $\sigma_G$ respects this decomposition. Let $K'=\text{exp}_G\mathfrak{k}'$, $T'=\text{exp}_G\mathfrak{t}'$ and $T=\text{exp}_G\mathfrak{t}$. $T$ is the maximal torus we will use from now on. Note that $T'$ is a maximal torus of $K'$. Let $W'$ be the Weyl group of $K'$ with respect to $T'$. Let $w_0'$ be the longest element in $W'$. Define $\sigma_+:\mathfrak{t}^*\to\mathfrak{t}^*$ by
	\[\sigma_+(\lambda)=-\sigma_{\mathfrak{g}}(w_0'\lambda)\]
and $\sigma_+: \mathfrak{t}\to \mathfrak{t}$ similarly. Let $k_0\in N_{(K')_0}(T')$ be any representative of $w_0'$. Let $R$ be the root system of $(\mathfrak{g}, \mathfrak{t})$. We define a positive root system $R_+$ as follows. Let 
\[R'=\{\alpha\in R|\sigma_{\mathfrak{t}^*}(\alpha)=\alpha\},\ \ R^\mathfrak{a}=\{\alpha|_\mathfrak{a}\ \ |\alpha\notin R'\}.\]
$R'$ is the root system of $(\mathfrak{k}', \mathfrak{t}')$, thus a root subsystem of $R$. $R^\mathfrak{a}$ is the system of restricted roots of the symmetric pair $(\mathfrak{g}, \mathfrak{t})$. Define
\[R_+:=R'_+\cup \{\alpha\in R|\ \ \alpha|_\mathfrak{a}\in R^\mathfrak{a}_+\}.\]
Now we can fix a choice of the closed Weyl alcove $\Delta$ with respect to $R_+$. 
\begin{proposition}\label{OsheaSjamaar}
	\begin{enumerate}
		\item\label{sigmapreserv}(\cite{OS}, Lemma 4.7(i)) $\sigma_+$ is an involution and $\sigma_+(R_+)=R_+$. Hence $\sigma_+$ preserves $\mathfrak{t}_+^*$ and $\Delta^k$. 
		\item\label{RRRRHR} (\cite{OS}, Addendum 4.11) We have $\sigma_G^*\overline{V_\lambda}\cong V_{\sigma_+(\lambda)}$, where $V_\lambda$ is the irreducible complex $G$- representation with highest weight $\lambda$. $V_\lambda$ is in $RR(G, \mathbb{R})$ (resp. $RH(G, \mathbb{R})$) iff $\sigma_+(\lambda)=\lambda$ and $(k_0^2)^\lambda=1$ (resp. $(k_0^2)^\lambda=-1$). $V_{\lambda}\oplus V_{\sigma_+(\lambda)}$ is in $RR(G, \mathbb{C})$ iff $\sigma_+(\lambda)\neq \lambda$. 
		\item (\cite{OS}, Lemma 4.7(ii)) $\sigma_+(\lambda)=\lambda$ iff either $\lambda\in\mathfrak{a}^*$, or $\lambda\in(\mathfrak{t}')^*$ and $\lambda=-w_0'\lambda$.
		\item (\cite{OS}, Lemma 4.7(iii)) If $\lambda\in\mathfrak{a}^*$, then $(k_0^2)^\lambda=1$. If $\lambda\in(\mathfrak{t}')^*$ and $\lambda=-w_0'\lambda$, then $(k_0^2)^\lambda=\pm 1$.
		\item\label{invonflag} Let 
		\begin{align*}
			\sigma_{G/T}: G/T&\to G/T\\
			gT&\mapsto \sigma_G(g)k_0^{-1}T.
		\end{align*}
		Then the Weyl covering map 
		\begin{align*}
			(G/T, \sigma_{G/T})\times(T, \sigma_+)&\to G^-\\
			(gT, t)&\mapsto gtg^{-1}
		\end{align*}
		is a Real map.
	\end{enumerate}
\end{proposition}
\begin{proof}
	We only show the last part. Let $\xi=\xi_1+\xi_2$ with $\xi_1\in\mathfrak{t}'$ and $\xi_2\in\mathfrak{a}$. Then
	\begin{align*}
		&\sigma_G(g)k_0^{-1}\text{exp}(\sigma_+(\xi_1+\xi_2))k_0\sigma_G(g)^{-1}\\
		=&\sigma_G(g)\text{exp}(w_0'^{-1}(-\sigma_\mathfrak{g}(w_0'(\xi_1+\xi_2))))\sigma_G(g)^{-1}\\
		=&\sigma_G(g)\text{exp}(-\xi_1+\xi_2)\sigma_G(g)^{-1}\\
		=&\sigma_G(g)\sigma_G(\text{exp}(\xi_1+\xi_2))^{-1}\sigma_G(g)^{-1}.
	\end{align*}
\end{proof}
\subsection{Construction of the equivariant Real fundamental DD bundle}\label{equivrealfundDD}
The equivariant fundamental DD-bundle $\mathcal{A}$ over $G$ was constructed in \cite{M1} using simplicial techniques. Let us first review the construction briefly and then point out how to equip $\mathcal{A}$ with a suitable Real structure so that it becomes an equivariant Real DD-bundle over $G^-$ with equivariant Real DD-class $[-\eta_G]\in H^3_{G\rtimes\Gamma}(G^-, \mathbb{Z}_\Gamma)$.

The compact Lie group $G$ admits the following simplicial description 
\[G=\coprod_I G/G_I\times \Delta_I/\sim\]
where, for $J\subset I$, 
\begin{eqnarray}\label{simrelation} (\varphi_I^J(gG_I), x)\sim (gG_I, \iota_J^I(x))\end{eqnarray}
with $\varphi_I^J: G/G_I\to G/G_J$ being the natural projection and $\iota_J^I:\Delta_J\to \Delta_I$ the inclusion of simplices. Let $\widetilde{G}_I$ be the universal cover of $G_I$, and define the central extension
	%\begin{proposition}[\cite{M1}]
	%	The central extension
		\[\widehat{G}_{I, t}:=\widetilde{G}_I\times_{\pi_1(G_I)}U(1)\]
of $G_I$, where %$\widetilde{G}_I$ is the universal cover of $G_I$, $t=\text{exp}(\xi)$ for $\xi\in\text{int}(\Delta_I)$ and 
$\pi_1(G_I)=\Lambda/\Lambda_I$ acts on $U(1)$ by multiplication by 
		\[%\lambda\cdot t=t^{-B^\flat(\lambda)}=
		\text{exp}(-2\pi\sqrt{-1}B(\lambda, \xi))\]
		for some $\xi\in\text{int}(\Delta_I)$ %satisfies the properties stated above.
	%\end{proposition}
	%\begin{proof}
	%	The first two properties are easy. For the last one, we only need to note that if we choose $\widehat{T}$ to be the image of the map 
	%	\begin{align*}
	%		i_I: T\times U(1)&\to \widehat{G}_{I, t}\\
	%		(\text{exp}_T(\zeta), z)&\mapsto [(\text{exp}_{\widetilde{G}_I}(\zeta), \text{exp}(-2\pi\sqrt{-1}B(\zeta, \xi))z)]
	%	\end{align*}
	%	then the conjugation action by a representative of $w\in W_I$ restricted to $\widehat{T}$ is indeed the one given in Proposition \ref{weylaction}.
	%\end{proof}
(the isomorphism class of the central extension $\widehat{G}_{I}$ is independent of the choice of $\xi$). 
%because $\text{exp}(\text{int}(\Delta_I))$ is contractible. We shall simply fix a choice of $t$ and drop $t$ from the subscript in the notation for the central extension of $G_I$ from now on. 
	By \cite[Lemma 3.5]{M1}, there exists a Hilbert space $\mathcal{H}$ equipped with unitary representations of the central extensions $\widehat{G}_I$ such that the central circle acts with weight $-1$ and, for $J\subset I$, the action of $\widehat{G}_J$ restricts to that of $\widehat{G}_I$. Putting $\mathcal{A}_I=G\times_{G_I}\mathcal{K}(\mathcal{H})$ and 
	\[\mathcal{A}=\coprod_I (\mathcal{A}_I\times \Delta_I)/\sim\]
	where the relation $\sim$ is similar to the one used in Equation (\ref{simrelation}), we get the desired fundamental DD-bundle.
	
	%We end this section by pointing out how to endow $\mathcal{A}$ with a suitable Real structure which descends to the anti-involution on $G$, as promised at the beginning of this section. 
	To equip $\mathcal{A}$ with a suitable Real structure, we first choose $T$ to be the distinguished maximal torus with respect to $\sigma_G$ as in Section \ref{distingtor}. %Let $\sigma_{G/G_I}: G/G_I\to G/G_{\sigma_+(I)}$ be defined by $gG_I\mapsto \sigma_G(g)k_0^{-1}G_{\sigma_+(I)}$. 
	By an easy generalization of (\ref{invonflag}) of Proposition \ref{OsheaSjamaar}, the involution on the simplicial pieces $G/G_I\times \Delta_I$ used in the simplicial description of $G$ defined by 
	\begin{align*}
		G/G_I\times\Delta_I&\to G/G_{\sigma_+(I)}\times \Delta_{\sigma_+(I)}\\
		(gG_I, \xi)&\mapsto (\sigma_G(g)k_0^{-1}G_{\sigma_+(I)}, \sigma_+(\xi))
	\end{align*}
	induces the anti-involution on $G$ (the involution $\sigma_+$ on $\Delta$, which permutes its subsimplices, induces an involution on the set of indices $\{I\}$, which by abuse of notation is also denoted by $\sigma_+$). Setting $\mathcal{H}':=\mathcal{H}\oplus\sigma_G^*\overline{\mathcal{H}}$ with the Real structure being swapping the two summands, we obtain a Real representation of $\widehat{G}_I$ with the central circle acting with weight $-1$, and $\mathcal{A}_I=G\times_{G_I}\mathcal{K}(\mathcal{H}')$. We find that the %simplicial piece $\mathcal{A}_I\times\Delta_I$ with 
	involution on the simplicial pieces
	\begin{align*}
		\mathcal{A}_I\times\Delta_I&\to \mathcal{A}_{\sigma_+(I)}\times\Delta_{\sigma_+(I)}\\
		([g, F], \xi)&\mapsto ([\sigma_G(g)k_0^{-1}, \overline{F}], \sigma_+(\xi))
	\end{align*}
	induces the desired Real structure on $\mathcal{A}$. 
\begin{remark}
	The other generator of $H^3_{G\rtimes\Gamma}(G^-, \mathbb{Z}_\Gamma)$ of infinite order, $[-\eta_G]+\text{DD}_\mathbb{R}(\pi^*o_{\mathbb{R}^{0, 4}})$, is the DD class of the same DD bundle $\mathcal{A}$ except that the involution $\sigma_\mathcal{A}$ is induced by the involution 
	\[([g, F], \xi)\mapsto ([\sigma_G(g)k_0^{-1}, J\overline{F}J^{-1}], \sigma_+(\xi))\]
	on the simplicial pieces $\mathcal{A}_I\times\Delta_I$, where $J$ is a `quaternionic quarter turn' on $\mathcal{H}$. 
\end{remark}

\section{A generalization of the Freed-Hopkins-Teleman in the Real context}\label{realFHT}
\subsection{The Freed-Hopkins-Teleman Theorem}

Let $\mathcal{A}$ be an equivariant DD bundle whose DD class is the generator of $H^3_G(G, \mathbb{Z})\cong\mathbb{Z}$. The equivariant twisted $K$-homology $K^G_*(G, \mathcal{A}^p)$ has a multiplicative structure induced by the cross product (see ({\ref{crossedprod}) in Section \ref{twistedkrhomology}).
\[K_*^G(G, \mathcal{A}^p)\otimes K^G_*(G, \mathcal{A}^p)\to K^G_*(G\times G, \pi_1^*\mathcal{A}^p\otimes\pi_2^*\mathcal{A}^p)\]
followed by the pushforward map induced by the group multiplication
\[m_*: K^G_*(G\times G, \pi_1^*\mathcal{A}^p\otimes\pi_2^*\mathcal{A}^p)\to K^G_*(G, \mathcal{A}^p)\]
Note that there is a Morita isomorphism $m^*\mathcal{A}^p\cong \pi_1^*\mathcal{A}^p\otimes\pi_2^*\mathcal{A}^p$ because $m^*[\eta_G]=\pi_1^*[\eta_G]+\pi_2^*[\eta_G]$. $m_*$ is independent of the equivariant Hilbert space bundle $\mathcal{E}$ on $G\times G$ which witnesses this Morita isomorphism, and hence canonically defined, since $H^2_G(G\times G, \mathbb{Z})=0$ (cf. (1) of Section \ref{twistedkrhomology}). The Freed-Hopkins-Teleman Theorem\footnote{Note that the Freed-Hopkins-Teleman Theorem deals with more general compact Lie groups and $\mathbb{Z}_2$-graded twists (see \cite{FHT3}).} asserts that

\begin{theorem}[The Freed-Hopkins-Teleman Theorem, \cite{Fr}, \cite{FHT1}, \cite{FHT2}, \cite{FHT3}]\label{FHT}
	Let $G$ be a simple and simply-connected Lie group, and $\textsf{h}^\vee$ be the dual Coxeter number of $G$ (for definition see Section \ref{notedefs}). The equivariant twisted $K$-homology $K^G_*(G, \mathcal{A}^{k+\textsf{h}^\vee})$ is isomorphic to the level $k$ Verlinde algebra $R_k(G)$ (to be explained below), for $k\geq 0$. More precisely, the pushforward map 
	\[\iota_*: R(G)\cong K^G_*(\text{pt})\to K^G_*(G, \mathcal{A}^{k+\textsf{h}^\vee})\]
	is onto with kernel being $I_k$, the level $k$ Verlinde ideal (for definition see Section \ref{notedefs}). 
\end{theorem}

The above Theorem merits some remarks. On one end of the isomorphism is the Verlinde algebra, which is an object of great interest in mathematical physics. It is the Grothendieck group of the positive energy representations of the level $k$ central extension of the (free) loop group $LG$, equipped with an intricately defined ring structure called the fusion product. It is known that $R_k(G)$, as an abelian group, is generated freely by the isomorphism classes of irreducible positive energy representations $V_\lambda$ with highest weight $\lambda$ in $\Lambda_k^*$, the set of level $k$ weights (see Section \ref{notedefs} for definition). The fusion product rule can be stipulated by defining its structural constants $c_{\lambda\mu}^{\gamma}$ with respect to those generators satisfying
\[[V_\lambda]\cdot[V_\mu]=\sum_{\gamma\in\Lambda_k^*}c_{\lambda\mu}^{\gamma}[V_\gamma]\]
to be the dimension of a certain vector space associated with the Riemann surface of genus 0, with three punctures labelled by $\lambda$, $\mu$ and $\gamma^*:=-w_0\gamma$. This vector space has its root in Conformal Field Theory (see \cite{V}) and can be interpreted as the space of conformal blocks (for one of its models see \cite{Be}), which was shown to be canonically isomorphic to the space of generalized theta functions of the moduli space of $G$-bundles (cf. \cite{BL} and the references therein). Thus one of the novelties of the Freed-Hopkins-Teleman Theorem is that it provides an algebro-topological perspective of the fusion product, in addition to the conformal field theory and algebro-geometric approaches. An alternative definition of the Verlinde algebra as a quotient ring of $R(G)$ is given in Section \ref{notedefs}. 

Let $\mathcal{A}$ be the equivariant Real fundamental DD bundle on $G^-$, whose construction is given in Section \ref{equivrealfundDD}. We will consider $KR^G_*(G^-, \mathcal{A}^{k+\textsf{h}^\vee})$ to formulate the Real version of FHT and give its proof in this Section. 

%Note that in this Real situation the pushforward map $m_*$ induced by the group multiplication is a canonical one because the equivariant Real cohomology $H^2_{G\rtimes \Gamma}(G^-\times G^-, \mathbb{Z}_\Gamma)=0$ (cf. (\ref{canonicalmorita}) in Section \ref{twistedkrhomology}), which can be proved along the lines of thought in Section \ref{MorG}.

\subsection{A structure theorem}
The following structure theorem is a generalization of \cite[Theorem 4.2]{Se}.
\begin{theorem}\label{strthm}
	Let $X$ be a Real $G$-space and $\mathcal{A}$ an equivariant Real DD bundle on $X$. Suppose $K^*_G(X, \mathcal{A})$ is a free abelian group which is decomposed by the involution $\sigma_{G}^*\circ\overline{\sigma}_{X}^*$ into the following summands
	\[K_G^*(X, \mathcal{A})=M_+\oplus M_-\oplus T\oplus\sigma_{G}^*\circ\overline{\sigma}_{X}^*T,\]
	where $\sigma_{G}^*\circ\overline{\sigma}_{X}^*$ is identity on $M_+$ and negation on $M_-$. Suppose further that there exist $x_1, \cdots, x_n\in KR_G^*(X, \mathcal{A})$ such that their images in $K_G^*(X, \mathcal{A})$ under the forgetful map (forgetting the Real structure) form a basis of $M_+\oplus M_-$. Let $F$ be the free $KR^*(\text{pt})$-module generated by $x_1, \cdots, x_n$, and $K^*(+)$ the complex $K$-theory of a point extended to a $\mathbb{Z}_8$-graded module over $K^*(\text{pt})\cong\mathbb{Z}$. Then the following map 
		\[f: F\oplus r(K^*(+)\otimes T)\to KR_G^*(X, \mathcal{A})\]
is an isomorphism of $KR^*(\text{pt})$-modules, where $r(x)=x+\sigma_G^*\circ\overline{\sigma}_X^*x\in KR^*_G(X, \mathcal{A})$ is the realification map defined in Proposition \ref{krprelim} (\ref{realcomplex}). 
\end{theorem}
\begin{proof} The proof will be a straightforward adaptation of that of \cite[Theorem 4.2]{Se} in the twisted equivariant setting. Consider the following $H(p, q)$ systems
	\begin{align*}
		HR_G^\alpha(p, q)(X, \mathcal{A})&:= KR_G^{-\alpha}(X\times S^{q, 0}, X\times S^{p, 0}, \pi^*\mathcal{A})\cong KR_G^{-\alpha+p}(X\times S^{q-p, 0}, \pi^*\mathcal{A})\\
		H_G^\alpha(p, q)(X, \mathcal{A})&:=K_G^{-\alpha}(X\times S^{q, 0}, X\times S^{p, 0}, \pi^*\mathcal{A})\cong K_G^{-\alpha+p}(X\times S^{q-p, 0}, \pi^*\mathcal{A})
	\end{align*}
	We use $\pi$ to denote various projection maps by abuse of notations, and $G$ acts on the spheres trivially. A generalization of the Gysin sequence (cf. \cite[Proposition 3.2]{At})
	\[\cdots\longrightarrow KR_G^{p-q}(X, \mathcal{A})\stackrel{\cdot(-\eta)^p}{\longrightarrow}KR_G^{-q}(X, \mathcal{A})\stackrel{\pi^*}{\longrightarrow}KR_G^{-q}(X\times S^{p, 0}, \pi^*\mathcal{A})\stackrel{\delta}{\longrightarrow}\cdots\]
	implies the following short exact sequence for $p\geq3$
	\[0\longrightarrow KR_G^{-q}(X, \mathcal{A})\stackrel{\pi^*}{\longrightarrow} KR_G^{-q}(X\times S^{p, 0}, \pi^*\mathcal{A})\stackrel{\delta}{\longrightarrow} KR_G^{p+1-q}(X, \mathcal{A})\longrightarrow 0\]
	If we take inverse limits as $p$ tends to infinity, then $\displaystyle\lim_{\stackrel{\longleftarrow}{p}} KR_G^{p+1-q}(X, \mathcal{A})=0$ and hence
	\[KR_G^*(X, \mathcal{A})=\lim_{\stackrel{\longleftarrow}{p}} KR^*_G(X\times S^{p, 0}, \pi^*\mathcal{A})\]
	This shows that the spectral sequence associated to the system $HR^*_G(p, q)(X, \mathcal{A})$ does converge to $KR_G^*(X, \mathcal{A})$. That the spectral sequence associated to the system $H^*_G(p, q)(X, \mathcal{A})$ converges to $K_G^*(X, \mathcal{A})$ follows from the proof of \cite[Lemma 4.1]{Se}. Now we define a map
	\begin{align*}
		f(p, q): HR^\alpha(p, q)(\text{pt})\otimes_{KR^*(\text{pt})}F\oplus r(H^\alpha(p, q)(\text{pt})\otimes T)&\to HR_G^\alpha(p, q)(X, \mathcal{A})\\
		\rho_1\otimes a_1\oplus r(\rho_2\otimes a_2)&\mapsto \rho_1a_1+r(\rho_2a_2).
	\end{align*}
	As $F$ is a free $KR^*(\text{pt})$-module and $T$ a free abelian group, and tensoring free objects and taking cohomology commute, $f$ is the abutment of $f(p, q)$. On the $E^{p, q}_1$-page, $f(p, q)$ becomes
	\[f_1^{p, q}: KR^{-q}(S^{1, 0})\otimes_{KR^*(\text{pt})}F\oplus r(K^{-q}(S^{1, 0})\otimes T)\to KR_G^{-q}(X\times S^{1, 0}, \pi^*\mathcal{A}).\]
	Note that $KR^{-q}(S^{1, 0})\cong K^{-q}(\text{pt})$, $K^{-q} (S^{1, 0})\cong K^{-q}(\text{pt})\oplus K^{-q}(\text{pt})$ and $KR_G^{-q}(X\times S^{1, 0}, \pi^*\mathcal{A})\cong K_G^{-q}(X, \mathcal{A})$. With the above identifications, 
	\begin{align*}
		r: K^{-q}(S^{1, 0})\otimes T&\to KR_G^{-q}(X\times S^{1, 0}, \pi^*\mathcal{A})\cong K_G^{-q}(X, \mathcal{A})\\
		(u_1, 0)\otimes t_1+(0, u_2)\otimes t_2&\mapsto u_1t_1+\overline{u_2}\sigma_G^*\circ\sigma_X^*\overline{t_2}.
	\end{align*} 
	So $r(K^*(S^{1, 0})\otimes T)=T\oplus \sigma_G^*\circ\overline{\sigma}_X^*T$. Together with the identification $KR^*(S^{1, 0})\otimes_{KR^*(\text{pt})}F\cong M_+\oplus M_-$, we have that $f_1^{p, q}$ is an isomorphism, and so is its abutment $f$. 
\end{proof}
\begin{corollary}\label{strthmhomology}
	Theorem \ref{strthm} still holds if $X$ is assumed to be a Real $G$-manifold which is Real orientable, twisted equivariant $KR$-theory is replaced by twisted equivariant $KR$-homology, $\sigma_X^*$ by $\sigma_{X*}$ throughout its statement.
\end{corollary}
\begin{proof}
	Let $X$ be Real $(p, q)$-orientable. By abuse of notation, we use pd to denote the Poincar\'e duality map for both the complex $K$-theory $K_G^*(X, \mathcal{A})\to K^G_{p+q-*}(X, \mathcal{A}^\text{opp}\otimes o_{TX})$ and $KR$-theory $KR_G^*(X, \mathcal{A})\to KR^G_{q-p-*}(X, \mathcal{A}^{\text{opp}}\otimes o_{TX})$. We shall show that the following square
	\begin{eqnarray}
		\xymatrix{K_G^*(X, \mathcal{A})\ar[r]^{\sigma_G^*\circ\pm\overline{\sigma}_X^*}\ar[d]^{\text{pd}}&K_G^*(X, \mathcal{A})\ar[d]^{\text{pd}}\\ K^G_{p+q-*}(X, \mathcal{A}^\text{opp}\otimes o_{TX})\ar[r]^{\sigma_G^*\circ\overline{\sigma}_{X*}}&K^G_{p+q-*}(X, \mathcal{A}^\text{opp}\otimes o_{TX})}
	\end{eqnarray}
	commutes, implying that the map pd for complex $K$-theory respects the decomposition of $K_G^*(X, \mathcal{A})$ and $K^G_*(X, \mathcal{A})$ by $\sigma_G\circ\overline{\sigma}^*_X$ and $\sigma_G^*\circ\overline{\sigma}_{X*}$ respectively as in Theorem \ref{strthm} (here the choice of the sign of the top map depends on whether $\sigma_X$ is orientation preserving or not). %In this way we will get a structural description of $KR^G_*(X, \mathcal{A})$ analogous to the one given in Theorem \ref{strthm} by applying pd for $KR$-theory to the structural description of $KR_G^*(X, \mathcal{A})$. 
The Corollary then follows from applying Poincar\'e duality for $KR$-theory to the structural description of $KR_G^*(X, \mathcal{A})$ given in Theorem \ref{strthm}. That $\sigma_G^*$ and complex conjugation commute with $\text{pd}$ being obvious, it remains to show that $\text{pd}\circ\pm\sigma_X^*=\sigma_{X*}\circ\text{pd}$. On the one hand, the wrong-way map $\sigma_{X*}$ on $K_G^*(X, \mathcal{A})$ satisfies $\text{pd}\circ\sigma_{X*}=\sigma_{X*}\circ \text{pd}$ by definition. On the other hand, we have the push-pull formula
	\[\sigma_{X*}(\sigma_X^*(x)\cdot y)=x\cdot\sigma_{X*}(y)\]
	for $x\in K_G^*(X, \mathcal{A})$ and $y\in K_G^*(X)$. Letting $y=1$, we get $\sigma_{X*}(1)=\text{pd}^{-1}\circ\sigma^*_X\circ\text{pd}(1)=\text{pd}^{-1}\circ\sigma^*_X([M])=\pm 1$, the choice of sign depending on whether $\sigma_X$ is orientation preserving. The push-pull formula then yields that, on $K_G^*(X, \mathcal{A})$, $\sigma_{X*}=\pm(\sigma_X^*)^{-1}=\pm\sigma_X^*$, concluding the proof.
\end{proof}

\subsection{The module structure of $KR^G_*(G^-, \mathcal{A}^{k+\textsf{h}^\vee})$}
\begin{lemma}\label{preservlevelk}
	The set of level $k$ weights $\Lambda_k^*$ is invariant under $\sigma_+$.
\end{lemma}
\begin{proof}
	This follows from Proposition \ref{OsheaSjamaar} (\ref{sigmapreserv}) and the fact that $\Lambda_k^*=\Delta^k\cap\Lambda^*$.
\end{proof}

\begin{proposition}\label{FHTdecomp}
	Let $P\subset \Lambda_k^*$ be a set of representatives of the 2-element orbits of the action of $\sigma_+$ on $\Lambda_k^*$. The involution $\overline{\sigma}_{G}^*\circ\sigma_{G^-*}$ decomposes $K^G_*(G^-, \mathcal{A}^{k+\textsf{h}^\vee})\cong R_k(G)$ into the following summands
	\[\bigoplus_{\{\lambda\in\Lambda^*_k|\sigma_+(\lambda)=\lambda\}}\mathbb{Z}[V_\lambda]\oplus\bigoplus_{\nu\in P}\mathbb{Z}[V_\nu]\oplus \bigoplus_{\nu\in P}\mathbb{Z}[V_{\sigma_+(\nu)}].\]
\end{proposition}
\begin{proof}
	It suffices to show that $\overline{\sigma}_{G}^*\circ\sigma_{G^-*}(V_\lambda)=V_{\sigma_+(\lambda)}$. The pushforward map $\iota_*: K^G_*(\text{pt})\to K^G_*(G, \mathcal{A}^{k+\textsf{h}^\vee})$, where $\text{pt}\in G$ is the identity element, can be easily seen to satisfy $\overline{\sigma}_{G}^*\circ\sigma_{G^-*}\circ\iota_*=\overline{\sigma}_G^*\circ\iota_*=\iota_*\circ\overline{\sigma}_{G}^*$. Note that for $\lambda\in\Lambda_k^*$, $W_\lambda\in K^G_*(\text{pt})$, the irreducible representation of $G$ with highest weight $\lambda$, is sent by $\iota_*$ to $V_\lambda\in K^G_*(G, \mathcal{A}^{k+\textsf{h}^\vee})$, the irreducible positive energy representation of $LG$ with highest weight $\lambda$. It follows that
	\begin{align*}
		\overline{\sigma}_{G}^*\circ\sigma_{G^-*}(V_\lambda)&=\iota_*\circ\overline{\sigma}_{G}^*(W_\lambda)\\
		&=\iota_*(W_{\sigma_+(\lambda)})\\
		&=V_{\sigma_+(\lambda)}.
	\end{align*}
\end{proof}

\begin{remark}
	Recall that a level $k$ positive energy representation of $LG$ is the projective representation $V$ descending from a representation of $S^1\ltimes\widetilde{LG}_k$ (where $\widetilde{LG}_k$ is the level $k$ central extension of $LG$ by the circle group and the $S^1$-action on $\widetilde{LG}_k$ is induced by rotating the loops in $G$) such that the weights of the restricted $S^1$-representation are bounded below. If $LG$ is equipped with the involution
	\begin{align*}
		\sigma_{LG}: LG&\to LG\\
				\ell&\mapsto\sigma_G\circ\ell\circ \delta
	\end{align*}
	where $\delta$ is the reflection on the unit circle, then it induces an involution on $S^1\ltimes\widetilde{LG}_k$ and $\sigma_{LG}^*\overline{V}$ is also a positive energy representation (note that the weights of the $S^1$-action on $\overline{V}$ are negative of those on $V$, so $\delta$ is put in place in the definition of $\sigma_{LG}$ to ensure that the weights of the $S^1$-action on $\sigma_{LG}^*\overline{V}$ are bounded below). One can see that $\sigma_{LG}^*\overline{V}_\lambda=V_{\sigma_+(\lambda)}$, and so by Proposition \ref{FHTdecomp}, the involution $\overline{\sigma}_{LG}^*$ on $R_k(G)$ corresponds to $\overline{\sigma}_G^*\sigma_{G^-*}$ on $K^G_*(G^-, \mathcal{A}^{k+\textsf{h}^\vee})$.
\end{remark}
\begin{corollary}\label{realFHTstrthm}
	Let $\mathcal{A}$ be the equivariant Real fundamental DD bundle over $G^-$ constructed in Section \ref{equivrealfundDD}. Using the same notations as in the previous Proposition, we have the following $KR_*(\text{pt})$-module isomorphism
	\[KR^G_*(G^-, \mathcal{A}^{k+\textsf{h}^\vee})\cong KR_*(\text{pt})\otimes \bigoplus_{\{\lambda\in\Lambda_k^*|\sigma_+(\lambda)=\lambda\}}\mathbb{Z}[V'_\lambda]\oplus r\left(K_*(+)\otimes\bigoplus_{\nu\in P}\mathbb{Z}[V_\nu]\right) \]
	Here $[V'_\lambda]$ is the element in $KR^G_*(G^-, \mathcal{A}^{k+\textsf{h}^\vee})$ which is mapped by the forgetful map to $[V_\lambda]\in K^G_*(G, \mathcal{A}^{k+\textsf{h}^\vee})$, and is assigned with degree 0 or $-4$ according as whether $(k_0^2)^\lambda=1$ or $-1$.
\end{corollary}
\begin{proof}
	We shall show that those $[V'_\lambda]\in KR^G_*(G^-, \mathcal{A}^{k+\textsf{h}^\vee})$ which are sent by the forgetful map to $V_\lambda\in K^G_*(G, \mathcal{A}^{k+\textsf{h}^\vee})$ exist. Consider the commutative diagram
	\begin{eqnarray}\label{commforget}
		\xymatrix{KR^G_*(\text{pt})\ar[r]^{\iota_*^\mathbb{R}}\ar[d]^c& KR^G_*(G^-, \mathcal{A}^{k+\textsf{h}^\vee})\ar[d]^{c}\\ K^G_*(\text{pt})\ar[r]^{\iota_*}& K^G_*(G, \mathcal{A}^{k+\textsf{h}^\vee})}
	\end{eqnarray}
	where the vertical maps are forgetful maps. Let $W'_\lambda\in KR^G_*(\text{pt})$ be the Real or Quaternionic representation (depending on whether $(k_0^2)^\lambda=1$ or $-1$, respectively) satisfying $c(W'_\lambda)=W_\lambda$. Let $V'_\lambda=\iota^\mathbb{R}_*(W'_\lambda)$. By the commutativity of the above diagram, $c(V'_\lambda)$ is indeed $V_\lambda$. Now we can apply Corollary \ref{strthmhomology} and Proposition \ref{FHTdecomp} to obtain the desired isomorphism.
\end{proof}
\begin{remark}\label{geometriccycles}
	In fact, $[V'_\lambda]$ and $[r(V_\mu)]$ can be realized, following \cite[Theorem 4.13]{M1}, as the image of the fundamental class of some conjagacy classes under the pushforward map induced by inclusion into $G^-$. Let $\lambda\in \Lambda^*_k$ such that $\sigma_+(\lambda)=\lambda$ and $\frac{B^\sharp(\lambda)}{k}\in \Delta_I$ and $\mathcal{C}^-_\lambda$ be the conjugacy class of $\text{exp}\left(\frac{B^\sharp(\lambda)}{k}\right)$ with the Real structure inherited from $a_G$ on $G^-$, i.e. the involution $gG_I\to \sigma_G(g)k_0^{-1}G_I$ if we identify $\mathcal{C}_\lambda^-$ with $G/G_I$\footnote{Note that in this case, $\sigma_+(I)=I$}. Let $(\iota_{\mathcal{C}^-_\lambda}, \mathcal{E}): (\mathcal{C}^-_\lambda, \text{Cl}(T\mathcal{C}^-_\lambda))\to (G^-, \mathcal{A}^{k+\textsf{h}^\vee})$ be the Morita morphism induced by inclusion, where 
	\[\mathcal{E}:=G\times_{G_I}(\mathcal{H}^k\otimes\mathbb{C}_\lambda\otimes\text{Hom}(\mathcal{H}^{\textsf{h}^\vee}, \textsf{S}_I))\]
	is the Hilbert space bundle on $G/G_I$ witnessing the Morita isomorphism $\text{Cl}(T\mathcal{C}^-_\lambda)\simeq \iota_{\mathcal{C}^-_\lambda}^*\mathcal{A}^{k+\textsf{h}^\vee}$
	with $\mathcal{H}$ being a Hilbert space with trivial $G_I$-action and $\textsf{S}_I$ the $G_I$-equivariant spinor module for $\text{Cl}(\mathfrak{g}_I^\perp)$ (cf. \cite[Equations (21), (26), (28)]{M1}). The Real or Quaternionic structure on $\mathcal{E}$ is given by 
	\[[gG_I, v\otimes z\otimes F]\mapsto [\sigma_G(g)k_0^{-1}G_I, \overline{v}\otimes\overline{z}\otimes\overline{F}]\] 
	The type of $\mathcal{E}$ (Real or Quaternionic bundle) depends on whether $(k_0^2)^\lambda=1$ or $-1$. The pushforward
	\[\iota_{\mathcal{C}^-_\lambda*}: KR^G_0(\mathcal{C}^-_\lambda, \text{Cl}(T\mathcal{C}^-_\lambda))\to KR^G_*(G^-, \mathcal{A}^{k+\textsf{h}^\vee})\]
	takes the fundamental class $[\mathcal{C}^-_\lambda]$ to $[V'_\lambda]$\footnote{The fundamental class $[\mathcal{C}_\lambda^-]$ is a degree 0 $KR$-homology class because $T\mathcal{C}^-_\lambda$ is a Real $(p, p)$-orientable vector bundle for some $p$ (cf. the anti-holomorphicity of the map $-\sigma_{\mathfrak{g}^*}$ on the coadjoint orbit $\mathcal{O}_\lambda$ proven in \cite[Proof of Theorem 4.8]{OS}). Its pushforward is of degree 0 or $-4$ according as $\mathcal{E}$ is Real or Quaternionic, i.e. $(k_0^2)^\lambda=1$ or $-1$.} (cf. \cite[Theorem 4.13]{M1}). For $\mu\in P$, $r(V_\mu)$ can be similarly realized as $\iota_{\mathcal{C}_\mu\cup\mathcal{C}_{\sigma_+(\mu)}*}([\mathcal{C}_\mu\cup\mathcal{C}_{\sigma_+(\mu)}])$ (note that the inherited Real structure on $\mathcal{C}_\mu\cup\mathcal{C}_{\sigma_+(\mu)}$ swaps the two components). 
\end{remark}
\subsection{A modified Pontryagin product on $KR^G_*(G^-, \mathcal{A}^{k+\textsf{h}^\vee})$}\label{modifiedPont}
Though the ring structure of $K^G_*(G, \mathcal{A}^{k+\textsf{h}^\vee})$ as in FHT is induced by the Pontryagin product, any possible ring structure of the Real version $KR^G_*(G^-, \mathcal{A}^{k+\textsf{h}^\vee})$ cannot be defined in this way because the group multiplication is not a Real map with respect to any group anti-involution. The situation can be remedied if we give $G\times G$ the following Real structure, which is inspired by the braid isomorphisms of quasi-Hamiltonian manifolds (\cite[Theorem 6.2]{AMM}).
\begin{definition}
	Let $\text{Ad}: G\times G\to G$ be the conjugation map $(g_1, g_2)\mapsto g_1g_2g_1^{-1}$ and $G^-\times G_\text{Ad}^-$ be the Real $G$-space $G\times G$ equipped with the involution
	\[(g_1, g_2)\mapsto (a_G(g_1), (a_G(\text{Ad}(g_1, g_2)))).\]
	and $G$-action by conjugation diagonally. 
\end{definition}
It can be easily checked that the group multiplication $m: G^-\times G^-_\text{Ad}\to G^-$ is a Real $G$-map. 
\begin{proposition}
The Real equivariant cohomology $H_{G\rtimes\Gamma}^3(G^-\times G^-_\text{Ad}, \mathbb{R}_\Gamma)$ is isomorphic to $\mathbb{R}^2$ and generated by $[\pi_1^*\eta_G]$ and $[\pi_2^*\eta_G]$, where $\pi_i: G\times G\to G$ is the projection onto the $i$-th factor.
\end{proposition}
\begin{proof}
	%Note that $H^3_{G\rtimes\Gamma}(G^-\times G_\text{Ad}^-, \mathbb{R}_\Gamma)$ is generated by $[\pi_i^*\eta_G]$, i=1, 2. 
	It suffices to show that both $[\pi_i^*\eta_G]$ are $\Gamma$-anti-invariant by Proposition \ref{antiinvcoh}. That $[\pi_1^*\eta_G]$ is $\Gamma$-anti-invariant can be proved as in the proof of Lemma \ref{invgamma}. As to the $\Gamma$-anti-invariance of $[\pi_2^*\eta_G]$, we first observe that
	\begin{align*}
		-(\gamma^*\pi_2^*\eta_G)(\xi)&=-(a_G\circ\text{Ad})^*(\eta_G(\sigma_G(\xi)))\\
							      &=-\text{Ad}^*(a_G)^*(\eta_G(\sigma_G(\xi)))\\
							      &=\text{Ad}^*(\eta_G(\xi))\ \ (\text{See the proof of Lemma \ref{invgamma}})\\
							      &=(\text{Ad}^*\eta_G)(\xi)
	\end{align*}
	It follows that $-\gamma^*[\pi_2^*\eta_G]=\text{Ad}^*[\eta_G]$. Expressing $\text{Ad}$ as the composition of the maps 
	\begin{align*}
		d: G\times G&\to G\times G\times G\\
		(g_1, g_2)&\mapsto (g_1, g_2, g_1^{-1})
	\end{align*}
	and the multiplication map $m: G\times G\times G\to G$, we have
	\begin{align*}
		\text{Ad}^*[\eta_G]&=d^*(m^*[\eta_G])\\
						&=d^*(\pi_1^*[\eta_G]+\pi_2^*[\eta_G]+\pi_3^*[\eta_G])\\
						&=(\pi_1\circ d)^*[\eta_G]+(\pi_2\circ d)^*[\eta_G]+(\pi_3^*\circ d)^*[\eta_G]\\
						&=\pi_1^*[\eta_G]+\pi_2^*[\eta_G]+\pi_1^*(\text{inv}^*[\eta_G])\\
						&=\pi_1^*[\eta_G]+\pi_2^*[\eta_G]-\pi_1^*[\eta_G]\\
						&=[\pi_2^*\eta_G].
	\end{align*}
\end{proof}
Since the pullback map $m^*: H^3_{G\rtimes\Gamma}(G^-, \mathbb{Z}_\Gamma)\to H^3_{G\rtimes\Gamma}(G^-\times G^-_\text{Ad}, \mathbb{Z}_\Gamma)$ takes $[\eta_G]$ to $[\pi_1^*\eta_G]+[\pi_2^*\eta_G]$, the Real DD bundle $m^*\mathcal{A}^{k+\textsf{h}^\vee}$ is Morita isomorphic to $\pi_1^*\mathcal{A}^{k+\textsf{h}^\vee}\otimes\pi_2^*\mathcal{A}^{k+\textsf{h}^\vee}$. Moreover, following the spectral sequence argument as in the proof of Proposition \ref{RealcohLie}, we have that $H^2_{G\rtimes\Gamma}(G^-\times G^-_\text{Ad}, \mathbb{Z}_\Gamma)=0$, so the pushforward map
\[m_*: KR^G_*(G^-\times G^-_\text{Ad}, \pi_1^*\mathcal{A}^{k+\textsf{h}^\vee}\otimes\pi_2^*\mathcal{A}^{k+\textsf{h}^\vee})\to KR^G_*(G^-, \mathcal{A}^{k+\textsf{h}^\vee})\]
is well-defined (cf. (\ref{canonicalmorita}) in Section \ref{twistedkrhomology}). We may now define the modified Pontraygin product on $KR^G_*(G^-, \mathcal{A}^{k+\textsf{h}^\vee})$ as the composition of the cross product
\[KR^G_*(G^-, \mathcal{A}^{k+\textsf{h}^\vee})\otimes KR^G_*(G^-, \mathcal{A}^{k+\textsf{h}^\vee})\to KR^G_*(G^-\times G^-_\text{Ad}, \pi_1^*\mathcal{A}^{k+\textsf{h}^\vee}\otimes\pi_2^*\mathcal{A}^{k+\textsf{h}^\vee})\]
and $m_*$. The cross product, more precisely, can be defined in this situation on the module generators $[V_\lambda]$ and $r(V_\mu)$ as the pushforward of the fundamental class of the product of relevant conjugacy classes (see Remark \ref{geometriccycles}). For example, the cross product of $[V'_{\lambda_1}]$ and $[V'_{\lambda_2}]$ is the image of the fundamental class under the map $\iota_{\mathcal{C}^-_{\lambda_1}\times\mathcal{C}^-_{\lambda_2, \text{Ad}}*}: KR^G_*(\mathcal{C}^-_{\lambda_1}\times\mathcal{C}^-_{\lambda_2, \text{Ad}}, \text{Cl}(T\mathcal{C}^-_{\lambda_1}\otimes T\mathcal{C}^-_{\lambda_2, \text{Ad}}))\to KR^G_*(G^-\times G^-_\text{Ad}, \pi_1^*\mathcal{A}^{k+\textsf{h}^\vee}\otimes\pi_2^*\mathcal{A}^{k+\textsf{h}^\vee})$.

\subsection{The Real Freed-Hopkins-Teleman Theorem}
\begin{definition}
	Let $RR_k(G)$ be the level $k$ Real Verlinde algebra, the Grothendieck group of isomorphism classes of level $k$ Real positive energy representations with respect to the loop group involution $\sigma_{LG}$ with the ring structure inherited from the fusion product structure of $R_k(G)$.
\end{definition}
\begin{theorem}\label{mainthm0}
	The modified Pontryagin product structure of $KR^G_*(G^-, \mathcal{A}^{k+\textsf{h}^\vee})$ (see Section \ref{modifiedPont}) on its $KR_*(\text{pt})$-module generators in Corollary \ref{realFHTstrthm} (i.e. $[V'_\lambda]$ for $\lambda\in\Lambda_k^*$ and $\sigma_+(\lambda)=\lambda$, and $r([V_\nu]\otimes\beta^i)$ for $\nu\in P$) is inherited from the fusion product structure of $R_k(G)$. In particular, the Real Verlinde algebra $RR_k(G)$ is ring isomorphic to $KR_0^G(G^-, \mathcal{A}^{k+\textsf{h}^\vee})$. 
\end{theorem}
\begin{proof}
We first prove that $RR_k(G)\cong KR_0^G(G^-, \mathcal{A}^{k+\textsf{h}^\vee})$. Similar to the Real representation ring $RR(G)$ (cf. Definition \ref{RRRing}), the Real Verlinde algebra is freely generated, as an abelian group, by irreducible Real positive energy representations of real, complex and quaternionic types, i.e. 
\[RR_k(G)=\bigoplus_{\{\lambda\in\Lambda_k^*|\sigma_+(\lambda)=\lambda, (k_0^2)^\lambda=1\}}\mathbb{Z}[V'_\lambda]\oplus\bigoplus_{\{\gamma\in\Lambda_k^*|\sigma_+(\gamma)=\gamma, (k_0^2)^\gamma=-1\}}\mathbb{Z}[V'_\gamma\oplus V'_\gamma]\oplus\bigoplus_{\nu\in P}r\left(\mathbb{Z}[V_\nu]\right)\] 
where by abuse of notation $V'_\lambda$ and $V'_\gamma\oplus V'_\gamma$ are irreducible Real positive energy representations of real and quaternionic types respectively (cf. Proposition \ref{RealQuatStr}(\ref{realtype}, \ref{quattype})). On the other hand, by Corollary \ref{realFHTstrthm}, the degree 0 piece $KR_0^G(G^-, \mathcal{A}^{k+\textsf{h}^\vee})$ is isomorphic to
\begin{align*}
	&\bigoplus_{\{\lambda\in\Lambda_k^*|\sigma_+(\lambda)=\lambda, (k_0^2)^\lambda=1\}}\mathbb{Z}[V'_\lambda]\oplus \bigoplus_{\{\gamma\in \Lambda_k^*|\sigma_+(\gamma)=\gamma, (k_0^2)^\gamma=-1\}}\mathbb{Z}[V'_\gamma]\otimes\mu\oplus\bigoplus_{\nu\in P}r\left(\mathbb{Z}[V_\nu]\right)\\
	\cong&RR_k(G)\ (\text{note that }[V'_\gamma]\otimes\mu=[V'_\gamma\oplus V'_\gamma] \text{ by Proposition \ref{KR_G}})
\end{align*}
as abelian groups. The forgetful map $c: KR_0^G(G^-, \mathcal{A}^{k+\textsf{h}^\vee})\to K^G_0(G, \mathcal{A}^{k+\textsf{h}^\vee})$ can be easily seen to be injective. Moreover, $c$ is a ring homomorphism because both the cross product and $m_*$ commute with the various forgetful maps from $KR$-homology to $K$-homology. So the ring structure of $KR_0^G(G^-, \mathcal{A}^{k+\textsf{h}^\vee})$ is also inherited from the fusion product structure of $K^G_0(G, \mathcal{A}^{k+\textsf{h}^\vee})\cong R_k(G)$. We have shown that $RR_k(G)$ and $KR^G_0(G^-, \mathcal{A}^{k+\textsf{h}^\vee})$ are ring isomorphic.

More generally, we shall show that the Pontrayagin product structure of $KR^G_*(G^-, \mathcal{A}^{k+\textsf{h}^\vee})$ among the $KR_*(\text{pt})$-module generators, namely, $\{[V'_\lambda]|\lambda\in\Lambda_k^*, \sigma_+(\lambda)=\lambda\}$ and $\{r([V_\nu]\otimes\beta^i|\nu\in P, 0\leq i\leq 3)\}$ is inherited from that of $K^G_*(G, \mathcal{A}^{k+\textsf{h}^\vee})$ (i.e. the fusion product structure of $R_k(G)$). It suffices to show the following cases.
\begin{enumerate}
	\item Let $[V'_\lambda]\in KR^G_0(G^-, \mathcal{A}^{k+\textsf{h}^\vee})$ and $[V'_\gamma]\in KR^G_{-4}(G^-, \mathcal{A}^{k+\textsf{h}^\vee})$. Then 
	\begin{align*}
		c([V'_\lambda]\cdot[V'_\gamma])&=[V_\lambda]\cdot[V_\gamma]\\
								 &=\sum_{\{\nu\in\Lambda_k^*|\sigma_+(\nu)=\nu\}}c_{\lambda\gamma}^\nu[V_\nu]+\sum_{\nu\in P}c_{\lambda\gamma}^\nu[V_\nu\oplus V_{\sigma_+(\nu)}]\\
								 &=\sum_{\{\nu\in\Lambda_k^*|\sigma_+(\nu)=\nu, (k_0^2)^\nu=1\}}\frac{c_{\lambda\gamma}^\nu}{2}c([V'_\nu]\otimes\mu)+\sum_{\{\nu\in\Lambda_k^*|\sigma_+(\nu)=\nu, (k_0^2)^\nu=-1\}}c_{\lambda\gamma}^\nu c([V_\nu])\\
								 &+\sum_{\nu\in P}c_{\lambda\gamma}^\nu c(r([V_\nu]\otimes\beta^2))\ (\text{cf. Proposition \ref{KR_G} and note that }[V'_\nu]\otimes\mu=[V'_\nu\oplus V'_\nu])
	\end{align*}
	It follows that $\displaystyle [V'_\lambda]\cdot[V'_\gamma]$ and $\displaystyle \sum_{\{\nu\in\Lambda_k^*|\sigma_+(\nu)=\nu, (k_0^2)^\nu=1\}}\frac{c_{\lambda\gamma}^\nu}{2}[V'_\nu]\otimes\mu+\sum_{\{\nu\in\Lambda_k^*|\sigma_+(\nu)=\nu, (k_0^2)^\nu=-1\}}c_{\lambda\gamma}^\nu [V_\nu]+\sum_{\nu\in P}c_{\lambda\gamma}^\nu r([V_\nu]\otimes\beta^2)$ differ by an element in the kernel of $c$, but using Corollary \ref{realFHTstrthm} $c$ is injective on $KR^G_{-4}(G^-, \mathcal{A}^{k+\textsf{h}^\vee})$ so in fact they are equal. The cases where both $[V'_\lambda]$ and $[V'_\gamma]$ are in $KR_0^G(G^-, \mathcal{A}^{k+\textsf{h}^\vee})$ or $KR_{-4}^G(G^-, \mathcal{A}^{k+\textsf{h}^\vee})$ can be dealt with similarly.
	\item For $[V'_\lambda]\in KR^G_{-4}(G^-, \mathcal{A}^{k+\textsf{h}^\vee})$ and $\nu\in P$ one can compute the product $[V'_\nu]\cdot r([V_\nu]\otimes\beta^i)$ as follows.
	\begin{align*}
		&[V'_\lambda]\cdot r([V_\nu]\otimes\beta^i)\\
		=&r(c([V'_\lambda])\cdot [V_\nu]\otimes\beta^i)\ (\text{by Proposition \ref{krprelim} (\ref{rcmix})})\\
		=&r([V_\lambda]\cdot[V_\nu]\otimes\beta^i)\\
		=&\sum_{\{\gamma\in\Lambda_k^*|\sigma_+(\gamma)=\gamma, (k_0^2)^\gamma=-1\}}c_{\lambda\nu}^{\gamma}r([V_\gamma]\otimes\beta^i)+\sum_{\{\gamma\in\Lambda_k^*|\sigma_+(\gamma)=\gamma, (k_0^2)^\gamma=1\}}c_{\lambda\nu}^\gamma r([V_\gamma]\otimes\beta^{i+2})\\
		&+\sum_{\gamma\in P}c_{\lambda\nu}^\gamma r([V_\gamma\oplus V_{\sigma_+(\gamma)}]\otimes\beta^i)\\
		=&\sum_{\{\gamma\in\Lambda_k^*|\sigma_+(\gamma)=\gamma, (k_0^2)^\gamma=-1\}}c_{\lambda\nu}^\gamma r(c([V'_\gamma])\otimes\beta^i)+\sum_{\{\gamma\in\Lambda_k^*|\sigma_+(\gamma)=\gamma, (k_0^2)^\gamma=1\}}c_{\gamma\nu}^\gamma r(c([V'_\gamma])\otimes\beta^{i+2})\\
		&+\sum_{\gamma\in P} c_{\lambda\nu}^\gamma r(c(r([V_\gamma]\otimes\beta^2))\otimes\beta^i)\\
		=&\sum_{\{\gamma\in\Lambda_k^*|\sigma_+(\gamma)=\gamma, (k_0^2)^\gamma=-1\}}c_{\lambda\nu}^\gamma[V'_\gamma]\cdot r(\beta^i)+\sum_{\{\gamma\in\Lambda_k^*|\sigma_+(\gamma)=\gamma, (k_0^2)^\gamma=1\}}c_{\lambda\nu}^\gamma[V'_\gamma]\cdot r(\beta^{i+2})+\\
		&\sum_{\gamma\in P} c_{\lambda\nu}^\gamma r([V_\gamma]\otimes\beta^2)\cdot r(\beta^i).
	\end{align*}
	The products $[V'_\lambda]\cdot r([V_\nu]\otimes\beta^i)$ for $[V'_\lambda]\in KR^G_0(G^-, \mathcal{A}^{k+\textsf{h}^\vee})$ and $r([V_{\nu_1}]\otimes\beta^i)\cdot r([V_{\nu_2}]\otimes\beta^j)$ can be similarly computed using the Propositions \ref{krprelim} (\ref{rcmix}), \ref{KR_G} and the fusion product structure of $R_k(G)$. This finishes the proof.
\end{enumerate} 
\end{proof}
We would also like to describe the ring structure of $KR^G_*(G^-, \mathcal{A}^{k+\textsf{h}^\vee})$ by expressing it as the quotient of the coefficient ring $KR^G_*(\text{pt})$ by an ideal which we term \emph{the level k Real Verlinde ideal}. First we have the 
\begin{lemma}\label{I_kinvt}
	We have $\overline{\sigma_G^*}I_k=I_k$.
\end{lemma}
\begin{proof}
	%First we shall show that $\overline{\sigma_G^*}I_k=I_k$. 
	Recall that $I_k$ is the vanishing ideal in $R(G)$ (viewed as a character ring) of the set
	\[\left\{\left.\text{exp}_T\left(B^\sharp\left(\frac{\lambda+\rho}{k+\textsf{h}^\vee}\right)\right)\right| \lambda\in \Lambda_k^*\right\}.\]
	Suppose that $\chi_V\in I_k$. For $\lambda\in\Lambda_k^*$, we have
	\begin{align*}
		&\chi_{\overline{\sigma_G^*}V}\left(\text{exp}_T\left(B^\sharp\left(\frac{\lambda+\rho}{k+\textsf{h}^\vee}\right)\right)\right)\\
		=&\chi_{\overline{\sigma_G^*}V}\left(k_0\text{exp}_T\left(B^\sharp\left(\frac{\lambda+\rho}{k+\textsf{h}^\vee}\right)\right)k_0^{-1}\right)\\
		=&\chi_{\overline{\sigma_G^*}V}\left(\text{exp}_T\left(B^\sharp\left(\frac{w'_0(\lambda+\rho)}{k+\textsf{h}^\vee}\right)\right)\right)\\
		=&\chi_V\left(\text{exp}_T\left(B^\sharp\left(\frac{-\sigma_{\mathfrak{g}^*}(w'_0(\lambda+\rho))}{k+\textsf{h}^\vee}\right)\right)\right)\\
		=&\chi_V\left(\text{exp}_T\left(B^\sharp\left(\frac{\sigma_+(\lambda)+\rho}{k+\textsf{h}^\vee}\right)\right)\right)\ \ (\sigma_+(\rho)=\rho\text{ by Proposition }\ref{OsheaSjamaar} (\ref{sigmapreserv}))\\
		=&0\ \ (\sigma_+(\lambda)\text{ is also a level }k\text{ weight by Lemma }\ref{preservlevelk}).
	\end{align*}
	This shows that $\overline{\sigma_G^*}I_k\subseteq I_k$. Applying $\overline{\sigma_G^*}$ to both sides yields the reverse inclusion, and hence the equality.
\end{proof}
\begin{definition}
	Let $RI_k$ be the kernel of the map $\iota_*^\mathbb{R}: KR^G_*(\text{pt})\to KR^G_*(G^-, \mathcal{A}^{k+\textsf{h}^\vee})$. We call $RI_k$ the \emph{level k Real Verlinde ideal}.
\end{definition}
\begin{theorem}\label{mainthm}
	\begin{enumerate}
		\item\label{generalVerlindeIdeal} Let $G$ be a simple, simply-connected compact Lie group. Let $\mathcal{A}$ be the Real fundamental DD bundle over $G^-$ constructed in Section \ref{equivrealfundDD}. Furthermore, by abuse of notation, let 
		\[c: \widetilde{R}:=RR(G, \mathbb{R})\oplus RH(G, \mathbb{R})\oplus RR(G, \mathbb{C})\to R(G)\]
		be the map forgetting Real or Quaternionic structures. 
	%Let the level $k$ Verlinde ideal $I_k$ be generated by $r_1, \cdots, r_m\in R(G)$, and $RI_k$ be the ideal in $KR^G_*(\text{pt})$ with generators obtained from $r_1, \cdots, r_m$ by the following way.
	%\begin{enumerate}
	%	\item Assigning each irreducible component of $r_i$ which is not in $R(G, \mathbb{C})$ with degree 0 (resp. $-4$) according as whether it can be made a Real representation (resp. Quaternionic representation), and
	%	\item replacing each irreducible component $s$ of $r_i$ which is in $R(G, \mathbb{C})$ with the double $s+\overline{\sigma_G^*}s$, which is assigned with degree 0.
	%\end{enumerate}
	Then the pushforward map 
	\[\iota_*^\mathbb{R}: KR^G_*(\text{pt})\to KR^G_*(G^-, \mathcal{A}^{k+\textsf{h}^\vee})\]
	is onto and its kernel, $RI_k$, is the ideal in $KR^G_*(\text{pt})$ defined by
	\[(c^{-1}(I_k), r(I_k\otimes\beta^i), i=1, 3).\]
	%where $c: RR(G, \mathbb{R})\oplus RH(G, \mathbb{R})\oplus RR(G, \mathbb{C})\to R(G)$ is the map forgetting Real or Quaternionic structures. 
		\item\label{VerlindeIdealGenerators} Let $I_k=(\rho_1, \cdots, \rho_n, \rho_{n+1}, \cdots, \rho_m)\subseteq R(G)$, where $\rho_j\cong \sigma_G^*\overline{\rho}_j$ for $1\leq j\leq n$, and $\rho_l\ncong\sigma_G^*\overline{\rho}_l$ for $n+1\leq l\leq m$. Let $\{\omega_1, \cdots, \omega_{2p}, \omega_{2p+1}, \cdots, \omega_q\}$ be the set of fundamental representations such that $\sigma_G^*\overline{\omega}_i\cong \omega_{p+i}$ for $1\leq i\leq p$, $\sigma_G^*\overline{\omega}_j\cong\omega_j$ for $2p+1\leq j\leq q$.\footnote{By \cite[Lemma 5.5]{Se}, $\overline{\sigma_G^*}$ permutes the set of fundamental representations.} Let 
		\[S=\{1\}\cup\{\text{square-free monomials }\alpha=\omega_{i_1}\omega_{i_2}\cdots\omega_{i_d}\text{ such that }\alpha\text{ and }\sigma_G^*\overline{\alpha}\text{ are coprime}\}.\]
		Then we have
		\begin{align*}
			&RI_k\\
			=&(c^{-1}(\rho_j), c^{-1}(\rho_l\tau+\sigma_G^*(\overline{\rho}_l\overline{\tau})), r(\rho_l\chi\otimes\beta^i),\ \text{for }\tau, \chi\in S,\ i=1, 3,\ 1\leq j\leq n,\ n+1\leq l\leq m)
		\end{align*}
		In particular, if $RR(G, \mathbb{C})=0$, then $c$ is a ring isomorphism and $RI_k=(c^{-1}(I_k))=(c^{-1}(\rho_j), 1\leq j\leq n)$.
	\end{enumerate}
\end{theorem}
\begin{proof}
	\begin{enumerate}
		\item  We have that $V'_\lambda$ and $r(V_\mu\otimes\beta^i)$, which are generators of $KR^G_*(G^-, \mathcal{A}^{k+\textsf{h}^\vee})$ as a $KR_*(\text{pt})$-module by Corollary \ref{realFHTstrthm}, are images of $W'_\lambda$ and $r(W_\mu\otimes\beta^i)$ under $\iota_*^\mathbb{R}$, where $W'_\lambda\in KR^G_*(\text{pt})$ is the Real or Quaternionic irreducible representation with highest weight $\lambda$, and $W_\mu\in K^G_*(\text{pt})$ the complex irreducible representation with highest weight $\mu$. Thus $\iota_*^\mathbb{R}$ is onto. 
	
	%Comparing the descriptions of the coefficient ring $KR^G_*(\text{pt})$ and $KR^G_*(G^-, \mathcal{A}^{k+\textsf{h}^\vee})$ in Proposition \ref{KR_G} and Corollary \ref{realFHTstrthm} respectively, and observing how the map $\iota^\mathbb{R}_*$ works yield that $RI_k$ indeed is the kernel of $\iota^\mathbb{R}_*$. 
	Next we shall show that $(c^{-1}(I_k), r(I_k)\otimes\beta^i, 1\leq i\leq 3)$ is indeed the kernel of $\iota^{\mathbb{R}}_*$. Suppose $\rho\in c^{-1}(I_k)\subseteq RR(G, \mathbb{R})\oplus RH(G, \mathbb{R})\oplus RR(G, \mathbb{C})\subseteq KR^G_0(\text{pt})\oplus KR^G_{-4}(\text{pt})$. Then by the commutative diagram (\ref{commforget}), 
	\[c\circ\iota^{\mathbb{R}}_*(\rho)=\iota_*\circ c(\rho)\in \iota_*I_k=\{0\}.\]
	So $\iota^\mathbb{R}_*(\rho)\in\text{ker}(c)$. Assume for the sake of contradiction that $\iota^{\mathbb{R}}_*(\rho)$ is not 0. On the one hand, by the Bott exact sequence (cf. \cite[Proposition 3.3]{At}) and Poincar\'e duality for twisted equivariant $KR$-theory, we have that $\iota^{\mathbb{R}}_*(\rho)$ is a multiple of $\eta$ and hence a nonzero 2-torsion element. On the other hand, according to Corollary \ref{realFHTstrthm}, $KR^G_0(G^-, \mathcal{A}^{k+\textsf{h}^\vee})\oplus KR^G_{-4}(G^-, \mathcal{A}^{k+\textsf{h}^\vee})$ is torsion-free and $\iota^{\mathbb{R}}_*(\rho)\in KR^G_0(G^-, \mathcal{A}^{k+\textsf{h}^\vee})\oplus KR^G_{-4}(G^-, \mathcal{A}^{k+\textsf{h}^\vee})$. 
	%of the following form
	%\[n_1\iota^{\mathbb{R}}_*(\rho_1)+n_2\iota^{\mathbb{R}}_*(\rho_2)\cdot\mu+n_3\iota^{\mathbb{R}}_*(r(\rho_3))+m_1\iota^{\mathbb{R}}_*(\xi_1)+m_2\iota^{\mathbb{R}}_*(\xi_2\cdot \mu)+m_3\iota^{\mathbb{R}}_*(r(\xi_3\otimes\beta^2))\]
	%where $n_i, m_i\in\mathbb{Z}$ are not all zero, $\rho_1, \rho_2\in RR(G, \mathbb{R})$, $\xi_1, \xi_2\in RH(G, \mathbb{R})$ and $\rho_3, \xi_3\in R(G, \mathbb{C})$ and their irreducible subrepresentations all have level $k$ weights as their highest weights. 
	So $\iota^{\mathbb{R}}_*(\rho)$ cannot be a nonzero 2-torsion, leading to a contradiction. We have shown that $c^{-1}(I_k)\subseteq RI_k$.	
	
	%Similarly if $\rho\in I_k$, then 
	%\begin{align*}
	%	c^{\text{twist}}\circ \iota^{\mathbb{R}}_*(r(\rho\otimes\beta^i))&=\iota_*\circ c(r(\rho\otimes\beta^i))\\
	%												    &=\iota_*(\rho\otimes\beta^i+(-1)^i\overline{\sigma_G^*}\rho\otimes\beta^i)\ \ (\text{By Proposition \ref{krprelim} (\ref{rcmix})})\\
	%												    &=0\ \ (\overline{\sigma_G^*}\rho\in I_k\text{ by the claim in the first paragraph of this proof.})
	%\end{align*}
	%So $\iota^{\mathbb{R}}_*(r(\rho\otimes\beta^i))\in \text{ker}(c^{\text{twist}})$. Following the argument in the previous paragraph, we again suppose for the sake of contradiction that $\iota^{\mathbb{R}}_*(r(\rho\otimes\beta^i))\neq 0$. When $i=2$, we may again use the previous argument to deduce that $\iota^{\mathbb{R}}_*(r(\rho\otimes\beta^2))$ is not 2-torsion and thus contradicting $\iota^{\mathbb{R}}_*(r(\rho\otimes\beta^2))\in\text{ker}(c^{\text{twist}})$. If $i=1$, then 
	%\[\iota^{\mathbb{R}}_*(r(\rho\otimes\beta))=\iota^{\mathbb{R}}_*(\rho'_1)\otimes\eta^2+\iota^{\mathbb{R}}_*(r(\rho_2\otimes\beta))\]
	Next, to show that $r(I_k\otimes\beta^i)\subseteq RI_k$, we first prove the commutativity of the following diagram.
	\begin{eqnarray}\label{rcommutesiota}
			\xymatrix{K^G_*(\text{pt})\ar[r]^{\iota_*}\ar[d]_r& K^G_*(G, \mathcal{A}^{k+\textsf{h}^\vee})\ar[d]^{r}\\ KR^G_*(\text{pt})\ar[r]^{\iota^\mathbb{R}_*}& KR^G_*(G^-, \mathcal{A}^{k+\textsf{h}^\vee})}
	\end{eqnarray}
	Let $x\in K^G_*(\text{pt})\cong R(G)$. Using the definition of $r$ given in Proposition \ref{krprelim} (\ref{realcomplex}), we have
	\begin{align*}
		r(\iota_*(\rho))&=\iota_*(x)+\sigma_G^*\circ\sigma_{G^-*}(\iota_*(\overline{x}))\\
		      		     &=\iota_*(x)+\sigma_G^*(\iota_*(\overline{x}))\ \ (\because a_G\circ\iota=\iota)\\
				     &=\iota_*(x)+\iota_*(\sigma_G^*\overline{x})\\
				     &=\iota_*^\mathbb{R}(x+\sigma_G^*(\overline{x}))\\
				     &=\iota_*^\mathbb{R}(r(x)).
	\end{align*}
	It follows that $\iota_*^\mathbb{R}(r(I_k\otimes\beta^i))=r(\iota_*(I_k\otimes\beta^i))=r(0)=0$ and thus $r(I_k\otimes\beta^i)\subseteq RI_k$. We get the inclusion $(c^{-1}(I_k), r(I_k\otimes\beta^i), 1\leq i\leq 3)\subseteq RI_k$. To prove the reverse inclusion, we first use the description of $KR^G_*(\text{pt})$ in Proposition \ref{KR_G} and let 
	\begin{align*}
		&\sum_{i=0}^2\chi_i\otimes \eta^i+\chi_3\otimes\mu+\sum_{j=0}^3r(\chi_{4, j}\otimes\beta^j)\in KR^G_*(\text{pt})\ \text{for }\chi_i\in RR(G, \mathbb{R})\oplus RH(G, \mathbb{R}),\\
		&0\leq i\leq 3\text{ and }\chi_{4, j}\in R(G, \mathbb{C})
	\end{align*}
	be in $RI_k$. We may rewrite the above expression as
	\begin{eqnarray}\label{genformKR_G}\chi_0+2\chi_3+\chi_{4, 0}+\sigma_G^*\overline{\chi}_{4, 0}+\chi_{4, 2}+\sigma_G^*\overline{\chi}_{4, 2}+\sum_{i=1}^2\chi_i\otimes\eta^i+\sum_{j=1, 3}r(\chi_{4, j}\otimes\beta^j)\end{eqnarray}
	by Proposition \ref{KR_G} (\ref{realquatexch}), where $2\chi_3, \chi_{4, 0}+\sigma_G^*\overline{\chi}_{4, 0}$ and $\chi_{4, 2}+\sigma_G^*\overline{\chi}_{4, 2}$ are equipped with suitable Real or Quaternionic structures as in Proposition \ref{RealQuatStr}. We may further assume, without loss of generality, that $\chi_{4, j}\ncong\sigma_{G}^*\overline{\chi}_{4, j}$ for $j=1, 3$, for if $\chi_{4, j}\cong\sigma_G^*\overline{\chi}_{4, j}$, then there is $\chi'_{4, j}\in RR(G)\oplus RH(G)$ such that $c(\chi'_{4, j})=\chi_{4, j}$, and by Proposition \ref{krprelim} (\ref{rcmix}), $r(\chi_{4, j}\otimes\beta^j)=r(c(\chi'_{4, j})\otimes\beta^j)=\chi'_{4, j}\otimes r(\beta^j)$, which can be assimilated into the 7th term of (\ref{genformKR_G}). Applying $\iota^\mathbb{R}_*$ to (\ref{genformKR_G}) and using the commutativity of the diagram (\ref{rcommutesiota}), we have
	\[\iota^\mathbb{R}_*(\chi_0+2\chi_3+\chi_{4, 0}+\sigma_G^*\overline{\chi}_{4, 0}+\chi_{4, 2}+\sigma_G^*\overline{\chi}_{4, 2})+\sum_{i=1}^2\iota^\mathbb{R}_*(\chi_i)\otimes\eta^i+\sum_{j=1, 3}r(\iota_*(\chi_{4, j})\otimes\beta^j)=0\]
	which holds if and only if 
	\[\iota^\mathbb{R}_*(\chi_0+2\chi_3+\chi_{4, 0}+\sigma_G^*\overline{\chi}_{4, 0}+\chi_{4, 2}+\sigma_G^*\overline{\chi}_{4, 2})=0,\ \iota^\mathbb{R}_*(\chi_i)=0,\ \iota_*(\chi_{4, j})=0,\ \text{for }1\leq i\leq 2,\ j=1, 3.\]
	It follows that $c(\chi_0+2\chi_3+\chi_{4, 0}+\sigma_G^*\overline{\chi}_{4, 0}+\chi_{4, 2}+\sigma_G^*\overline{\chi}_{4, 2})$ and $c(\chi_i)$, $1\leq i\leq 2$, are in $I_k$, and $\chi_{4, j}\in I_k$. We have completed the proof of the assertion that $RI_k=(c^{-1}(I_k), r(I_k\otimes\beta^i, i=1, 3))$.
		\item It suffices to show that the ideal described in item (\ref{generalVerlindeIdeal}) of Theorem \ref{mainthm} is in the ideal described in item (\ref{VerlindeIdealGenerators}), since the reverse inclusion is obvious. Note that any element in $I_k\subseteq R(G)$ is of the form
		\[\rho_1\tau_1+\cdots+\rho_m\tau_m\ \text{for }\tau_1, \cdots, \tau_m\in R(G).\]
		By Proposition \ref{KR_G}, any element in $c^{-1}(I_k)\subseteq \widetilde{R}$ is of the form
		\begin{align*}
			&c^{-1}(\rho_1)c^{-1}(\tau_1)+\cdots+c^{-1}(\rho_n)c^{-1}(\tau_n)+c^{-1}(\rho_{n+1}\tau_{n+1}+\sigma_G^*(\overline{\rho}_{n+1})\sigma_G^*(\overline{\tau}_{n+1}))\\
			&+\cdots+c^{-1}(\rho_m\tau_m+\sigma_G^*(\overline{\rho}_m)\sigma_G^*(\overline{\tau}_m))
		\end{align*}
		for $\tau_j\in R(G)^{\overline{\sigma_G^*}}$, $1\leq j\leq n$ and $\tau_l\in R(G)$, $n+1\leq l\leq m$ (note that by Lemma \ref{I_kinvt}, $\sigma_G^*(\overline{\rho_l})\in I_k$ since $\rho_l\in I_k$). The terms $c^{-1}(\rho_1)c^{-1}(\tau_1)+\cdots+c^{-1}(\rho_n)c^{-1}(\tau_n)$ lies in $(c^{-1}(\rho_1), \cdots, c^{-1}(\rho_n))$. As to the latter terms, say $c^{-1}(\rho_l\tau_l+\sigma_G^*(\overline{\rho}_l)\sigma_G^*(\overline{\tau}_l))$, we shall show that it is in $(c^{-1}(\rho_l\tau+\sigma_G^*(\overline{\rho}_l\overline{\tau})), n+1\leq l\leq m, \tau\in S)$. Since $G$ is simply-connected, $R(G)$ is isomorphic to $\mathbb{Z}[\omega_1, \cdots, \omega_q]$ and so $\tau$ can be written as a polynomial $P(\omega_1, \cdots, \omega_q)$. The monomial terms of $p(\omega_1, \cdots, \omega_q)$ can be classified into two types, namely, those which are invariant under $\overline{\sigma_G^*}$, and those which are not. If $\alpha$ is a monomial of the first type, then 
		\[c^{-1}(\rho_l\alpha+\sigma_G^*(\overline{\rho}_l\overline{\alpha}))=c^{-1}((\rho_l+\sigma_G^*\overline{\rho}_l)\alpha)\in (c^{-1}(\rho_l+\sigma_G^*\overline{\rho}_l)).\]
		If $\alpha$ is of the second type, then it is of the form $n\omega_{i_1}^{\nu_{i_1}}\cdots\omega_{i_d}^{\nu_{i_d}}\alpha'$ for some $n\in\mathbb{Z}$, where $\sigma_G^*\overline{\alpha'}\cong\alpha'$, no two of the indices $i_1, \cdots, i_d\in \{1, 2, \cdots, 2p\}$ differ by $p$, and $\nu_{i_1}, \cdots, \nu_{i_d}\geq 1$. Let $\sigma(i_j)=\begin{cases}i_j+p&\ \text{if }1\leq i_j\leq p\\ i_j-p&\ \text{if }p+1\leq i_j\leq 2p\end{cases}$. It suffices to show that $c^{-1}(\rho_l\omega_{i_1}^{\nu_{i_1}}\cdots\omega_{i_d}^{\nu_{i_d}}+\sigma_G^*\overline{\rho_l}\omega_{\sigma(i_1)}^{\nu_{i_1}}\cdots\omega_{\sigma(i_d)}^{\nu_{i_d}})\in (c^{-1}(\rho_l\tau+\sigma_G^*(\overline{\rho}_l\overline{\tau})),\ \tau\in S)$. If $\nu_{i_1}=\cdots=\nu_{i_d}=1$ then the assertion is obvious. If any one of $\nu_{i_1}, \cdots, \nu_{i_d}$ is greater than 1 (say $\nu_{i_1}>1$), then we write
		\begin{align}
			&c^{-1}(\rho_l\omega_{i_1}^{\nu_{i_1}}\cdots\omega_{i_d}^{\nu_{i_d}}+\sigma_G^*\overline{\rho}_l\omega_{\sigma(i_1)}^{\nu_1}\cdots\omega_{\sigma(i_d)}^{\nu_{i_d}}) \label{inversec}\\
			=&c^{-1}(\rho_l\omega_{i_1}+\sigma_G^*\overline{\rho}_l\omega_{\sigma(i_1)})c^{-1}(\omega_{i_1}^{\nu_{i_1}-1}\omega_{i_2}^{\nu_{i_2}}\cdots\omega_{i_d}^{\nu_{i_d}}+\omega_{\sigma(i_1)}^{\nu_{i_1}-1}\omega_{\sigma(i_2)}^{\nu_{(i_2)}}\cdots\omega_{\sigma(i_d)}^{\nu_{i_d}})-\nonumber\\
			&c^{-1}(\rho_l\omega_{\sigma(i_1)}^{\nu_{i_1}-2}\omega_{\sigma(i_2)}^{\nu_{i_2}}\cdots\omega_{\sigma(i_d)}^{\nu_{i_d}}+\sigma_G^*\overline{\rho}_l\omega_{i_1}^{\nu_{i_1}-2}\omega_{i_2}^{\nu_{i_2}}\cdots\omega_{i_d}^{\nu_{i_d}})c^{-1}(\omega_{i_1}\omega_{\sigma(i_1)}). \nonumber
		\end{align} 
		The first term lies in the ideal $(c^{-1}(\rho_l\tau+\sigma_G^*(\overline{\rho}_l\overline{\tau})),\ \tau\in S)$, so does the second term by inductive hypothesis on $\nu_{i_1}, \cdots, \nu_{i_d}$. The expresion (\ref{inversec}) indeed lies in $(c^{-1}(\rho_l\tau+\sigma_G^*(\overline{\rho}_l\overline{\tau})),\ \tau\in S)$. 
		
		It remains to show that $(r(I_k\otimes\beta^i),\ i=1, 3)$ lies in the ideal described in item (\ref{generalVerlindeIdeal}) of Theorem \ref{mainthm}. Any element of $r(I_k\otimes\beta^i)$ takes the form
		\[\sum_{\substack{i=1, 3\\ 1\leq k\leq m}}r(\rho_k\chi_k\otimes\beta^i)\text{ for }\chi_k\in R(G).\]
		For $1\leq j\leq n$, then there exists $\rho'_j\in \widetilde{R}$ such that $c(\rho'_j)=\rho_j$, and 
		\begin{align*}
			r(\rho_j\chi_j\otimes\beta^i)&=r(c(\rho'_j)\chi_j\otimes\beta^i)\\
								   &=\rho'_jr(\chi_j\otimes\beta^i)\ (\text{by Proposition \ref{krprelim} (\ref{rcmix})})\\
								   &\in (\rho'_j)\ (\text{or }(c^{-1}(\rho_j))).
		\end{align*}
		For the term $r(\rho_l\chi_l\otimes\beta^i)$ where $n+1\leq l\leq m$, one can prove that it lies in the ideal $(r(\rho_l\chi\otimes\beta^i),\ \chi\in S)$ by reprising the inductive argument used in the last paragraph. Again assume without loss of generality that $\chi=\omega_{i_1}^{\nu_{i_1}}\cdots\omega_{i_d}^{\nu_{i_d}}$ where no two of the indices $i_1, \cdots, i_d\in\{1, \cdots, 2p\}$ differ by $p$. Let $\nu_{i_1}>1$. The assertion then follows from the equation below.
		\begin{align*}
			&r(\rho_l\omega_{i_1}^{\nu_{i_1}}\cdots\omega_{i_d}^{\nu_{i_d}}\otimes\beta^i)\\
			=&r(\rho_l\omega_{i_1}^{\nu_{i_1}-1}\omega_{i_2}^{\nu_{i_2}}\cdots\omega_{i_d}^{\nu_{i_d}}\otimes\beta^i)(c^{-1}(\omega_{i_1}+\omega_{\sigma(i_1)}))-\\
			&-r(\rho_l\omega_{i_1}^{\nu_{i_1}-2}\omega_{i_2}^{\nu_{i_2}}\cdots\omega_{i_d}^{\nu_{i_d}}\otimes\beta^i)(c^{-1}(\omega_{i_1}\omega_{\sigma(i_1)}))\ (\text{by Proposition \ref{krprelim} (\ref{rcmix})})
		\end{align*}
	\end{enumerate}
\end{proof}

\begin{remark}\label{lastrmk}
	\begin{enumerate}
		\item In \cite{Dou}, Douglas gave an explicit description of the Verlinde algebra of $G$. In particular, he gave a list of generators of $I_k$ for each type of simple, connected, simply-connected and compact Lie groups in such a way that the number of generators is independent of the level $k$. Together with Theorem \ref{mainthm} and the description of the ring structure of the coefficient ring $KR^G_*(\text{pt})$ given in Proposition \ref{KR_G} (and which can also be deduced from Corollary \ref{strthmhomology}), one can obtain the ring structure of $KR^G_*(G^-, \mathcal{A}^{k+\textsf{h}^\vee})$ explicitly.
		%\item \label{rmk2} By Theorem \ref{mainthm}, the degree 0 part $KR_0^G(G^-, \mathcal{A}^{k+\textsf{h}^\vee})\cong RR(G)/(RI_k\cap RR(G))$ gives the \emph{Real Verlinde algebra}, the Grothendieck group of the isomorphism classes of Real positive energy representations of the Real loop group $LG$, where the Real structure of $LG$ is given by 
		%\begin{align*}
		%	\sigma_{LG}: LG&\to LG\\
		%	\ell&\mapsto\sigma_G\circ\ell\circ c
		%\end{align*}
		%with $c$ meaning reflection on the unit circle.
		\item In \cite{M1}, FHT was proved using Segal's spectral sequence (cf. \cite{S1} for its definition), which was shown to collapse on the $E_2$-page, which in turn amounts to a free resolution of the Verlinde algebra. In fact Segal's spectral sequence is a generalized form of Mayer-Vietoris sequence. The $G$-invariant open cover used in FHT consists of certain open subsets in $G$ which deformation retract to conjugacy classes $G$-equivariantly. In \cite{F2}, the twisted equivariant $KR$-homology of $G$ is computed in the case where $R(G, \mathbb{C})=0$, following the spectral sequence arguments in \cite{M1}. The condition that there is no complex type representation is put in place just to ensure that the open cover is also invariant under the anti-involution on $G$, so that the Segal's spectral sequence can be adapted easily in this Real context. It was found that the Real $E_2$-page is the tensor product of the free resolution of the Verlinde algebra in the proof of \cite{M1} and the coefficient ring $KR_*(\text{pt})$, by virtue of the equivariant Real formality. Due to the nulhomotopy of the free resolution of the Verlinde algebra, the Real spectral sequence also collapses on the Real $E_2$-page, whose homology is the tensor product of the Verlinde algebra and the coefficient ring. Thus in this case $KR^G_*(G^-, \mathcal{A}^{k+\textsf{h}^\vee})\cong R_k(G)\otimes KR_*(\text{pt})$, in agreement with Theorem \ref{mainthm}.
	\end{enumerate}
\end{remark}

\begin{example}
	Consider $G=SU(2n)$ equipped with the quaternionic involution $g\mapsto J_n\overline{g}J_n^{-1}$, where $J_n=\begin{pmatrix} & -I_n\\ I_n&\end{pmatrix}$. In this case $R(G, \mathbb{C})=0$. Let $\sigma_{2n}\in R(G)$ (resp. $\sigma'_{2n}\in RH(G, \mathbb{R})$) be the class of the standard representation of $G$ (resp. the class of the standard representation equipped with the Quaternionic structure). By Corollary \ref{strthmhomology} (or Proposition \ref{KR_G} and Poincar\'e duality), 
	\[KR_*^G(\text{pt})\cong\mathbb{Z}[\sigma'_{2n}, \bigwedge\nolimits^2\sigma'_{2n}, \cdots, \bigwedge\nolimits^{2n-1}\sigma'_{2n}]\otimes KR_*(\text{pt})\]
	 where $\bigwedge\nolimits^{2k}\sigma'_{2n}\in KR^G_0(\text{pt})$ and $\bigwedge\nolimits^{2k-1}\sigma'_{2n}\in KR^G_{-4}(\text{pt})$ for $1\leq k\leq n$. According to \cite[Theorem 1.1]{Dou} the level $k$ Verlinde ideal $I_k$ is 
	 \[(W_{(k+1)L_1+L_2+\cdots+L_{2n-1}}, W_{(k+1)L_1+L_2+\cdots+L_{2n-2}}, \cdots, W_{(k+1)L_1+L_2}, W_{(k+1)L_1})\]
	 where $L_i$ is the $i$-th standard weight of $G$. The Real Verlinde ideal $RI_k$ is 
	  \[(W'_{(k+1)L_1+L_2+\cdots+L_{2n-1}}, W'_{(k+1)L_1+L_2+\cdots+L_{2n-2}}, \cdots, W'_{(k+1)L_1+L_2}, W'_{(k+1)L_1})\]
	  The prime notation is added to indicate that the classes of representations are equipped with either Real or Quaternionic structure, which also determines their degrees. $W'_{(k+1)L_1+L_2+\cdots+L_i}\in KR^G_0(\text{pt})$ if $k$ and $i$ are of the same parity, and $W'_{(k+1)L_1+L_2+\cdots+L_i}\in KR^G_{-4}(\text{pt})$ if $k$ and $i$ are not. By Theorem \ref{mainthm}, 
	  \[KR^G_*(G^-, \mathcal{A}^{k+2n})\cong \frac{\mathbb{Z}[\sigma'_{2n}, \bigwedge\nolimits^2\sigma'_{2n}, \cdots, \bigwedge\nolimits^{2n-1}\sigma'_{2n}]\otimes KR_*(\text{pt})}{(W'_{(k+1)L_1+L_2+\cdots+L_{2n-1}}, W'_{(k+1)L_1+L_2+\cdots+L_{2n-2}}, \cdots, W'_{(k+1)L_1})}.\]
\end{example}

\begin{appendix}
\section{}	 
This Appendix is devoted to the background material for $KR$-theory and Real representation rings, which is taken directly from \cite{F}.
	\subsection{$KR$-theory}
	\begin{definition}\label{krtheory}
	\begin{enumerate}
		\item A \emph{Real space} is a pair $(X, \sigma_X)$ where $X$ is a~topological space equipped with an involutive homeomorphism $\sigma_X$, i.e. $\sigma_X^2=\text{Id}_X$. A \emph{Real pair} is a pair $(X, Y)$ where~$Y$ is a~closed subspace of~$X$ invariant under $\sigma_X$.
		\item Let $\mathbb{R}^{p, q}$ be the Euclidean space $\mathbb{R}^{p+q}$ equipped with the involution which is the identity on the first $q$ coordinates and negation on the last $p$-coordinates. Let $B^{p, q}$ and $S^{p, q}$ be the unit ball and sphere in $\mathbb{R}^{p, q}$ with the inherited involution.
		\item A \emph{Real vector bundle} (to be distinguished from the usual real vector bundle) over $X$ is a~complex vector bundle $E$ over $X$ which itself is also a~Real space with involutive homeomorphism $\sigma_E$ satisfying
			\begin{enumerate}
				\item $\sigma_X\circ p=p\circ\sigma_E$, where $p: E\to X$ is the projection map,
				\item $\sigma_E$ maps $E_x$ to $E_{\sigma_X(x)}$ anti-linearly.
			\end{enumerate}
A \emph{Quaternionic vector bundle}
(to be distinguished from the usual quaternionic vector bundle) over $X$ is a complex vector bundle $E$ over $X$ equipped with an anti-linear lift $\sigma_E$ of $\sigma_X$ such that $\sigma_E^2=-\text{Id}_E$.
		\item A complex of Real vector bundles over the Real pair $(X, Y)$ is a complex of Real vector bundles over $X$
		\[0\longrightarrow E_1\longrightarrow E_2\longrightarrow\cdots\longrightarrow E_n\longrightarrow 0\]
		which is exact on $Y$.
		\item Let $X$ be a~Real space. The ring $KR(X)$ is the Grothendieck group of the isomorphism classes of Real vector bundles over $X$, equipped with the usual product structure induced by tensor product of vector bundles over $\mathbb{C}$. The relative $KR$-theory for a Real pair $KR(X, Y)$ can be similarly defined using complexes of Real vector bundles over $(X, Y)$, modulo homotopy equivalence and addition of acyclic complexes (cf. \cite{S2}). In general, the graded $KR$-theory ring of the Real pair $(X, Y)$ is given by
		\begin{gather*}
			KR^*(X, Y):=\bigoplus_{q=0}^7 KR^{-q}(X, Y),
		\end{gather*}
where
		\begin{eqnarray}\label{higherkrdef}
			KR^{-q}(X, Y):=KR\big(X\times B^{0,q}, X\times S^{0, q}\cup Y\times B^{0,q}\big).
		\end{eqnarray}
The ring structure of $KR^*$ is extended from that of~$KR$, in a~way analogous to the case of complex $K$-theory.
The number of graded pieces, which is 8, is a result of Bott periodicity for $KR$-theory (cf. \cite{At}).
\end{enumerate}
\end{definition}

Note that when $\sigma_X=\text{Id}_X$, then $KR(X)\cong KO(X)$. On the other hand, if $X\times \mathbb{Z}_2$ is given the involution which swaps the two copies of $X$, then $KR(X\times\mathbb{Z}_2)\cong K(X)$. Also, if $X$ is equipped with the trivial involution, then $KR(X\times S^{2,0})\cong KSC(X)$, the Grothendieck group of homotopy classes of self-conjugate bundles over $X$ (cf.~\cite{At}). In this way, it is natural to view $KR$-theory as a~unifying thread of $KO$-theory, $K$-theory and $\mathit{KSC}$-theory.

On top of the Real structure, we may further add compatible group actions and define equivariant $KR$-theory.
\begin{definition}\label{eqkrtheory}
	\begin{enumerate}
		\item A \emph{Real} $G$-\emph{space} $X$ is a~quadruple $(X, G, \sigma_X, \sigma_G)$ where a group $G$ acts on $X$ and $\sigma_G$ is an involutive automorphism of~$G$ such that
			\begin{gather*}
				\sigma_X(g\cdot x)=\sigma_G(g)\cdot\sigma_X(x).
			\end{gather*}
		\item A \emph{Real} $G$ \emph{vector bundle} $E$ over a Real $G$-space $X$ is a Real vector bundle and a $G$-bundle over $X$, and it is also a Real $G$-space.
		\item In a similar spirit, one can def\/ine equivariant $KR$-theory $KR_G^*(X, Y)$. Notice that the $G$-actions on $B^{0,q}$ and $S^{0,q}$ in the def\/inition of $KR_G^{-q}(X, Y)$ (cf. Equation (\ref{higherkrdef})) are trivial.
	\end{enumerate}
\end{definition}

\begin{definition}
	\begin{enumerate}
		\item Let $K^*(+)$ be the complex $K$-theory of a point extended to a $\mathbb{Z}_8$-graded algebra over $K^0(\text{pt})\cong\mathbb{Z}$, i.e. $\displaystyle K^*(+)\cong\mathbb{Z}[\beta]\left/\beta^4-1\right.$. Here $\beta\in K^{-2}(+)$, called the \emph{Bott class}, is the class of the reduced canonical line bundle on $\mathbb{CP}^1\cong S^2$.
		\item Let $\overline{\sigma_X^*}$ be the map defined on (equivariant) vector bundles on $X$ by $\overline{\sigma_X^*}E:=\sigma_X^*\overline{E}$. The involution induced by $\overline{\sigma_X^*}$ on $K^*_G(X)$ is also denoted by $\overline{\sigma_X^*}$ for simplicity.
	\end{enumerate}
\end{definition}

In the following proposition, we collect, for reader's convenience, some basic results of $KR$-theory (cf. \cite[Section 2]{Se}), some of which are stated in the more general context of equivariant $KR$-theory.

\begin{proposition}\label{krprelim}
	\begin{enumerate}
		\item\label{krcoeffring} We have
			\begin{gather*}
				KR^*(\text{pt})\cong\mathbb{Z}[\eta, \mu]\big/\big(2\eta, \eta^3, \mu\eta, \mu^2-4\big) ,
			\end{gather*}
where $\eta\in KR^{-1}(\text{pt})$, $\mu\in KR^{-4}(\text{pt})$ are represented by the reduced Hopf bundles of $\mathbb{R}\mathbb{P}^1$ and $\mathbb{H}\mathbb{P}^1$ respectively.
		\item\label{realcomplex} Let $c: KR^*_G(X)\to K^*_G(X)$ be the homomorphism which forgets the Real structure of Real vector bundles, and $r: K^*_G(X)\to KR^*_G(X)$ be the realification map defined by $[E]\mapsto [E\oplus\sigma_G^*\sigma_X^*\overline{E}]$, where $\sigma_G^*$ means twisting the original $G$-action on $E$ by $\sigma_G$ and the Real structure of $E\oplus\sigma_G^*\sigma_X^*\overline{E}$ is given by swapping the summands.
Then we have the following relations
		\begin{enumerate}
			\item $c(1)=1$, $c(\eta)=0$, $c(\mu)=2\beta^2$, where $\beta\in K^{-2}(\text{pt})$ is the Bott class,
			\item $r(1)=2$, $r(\beta)=\eta^2$, $r(\beta^2)=\mu$, $r(\beta^3)=0$,
			\item\label{rcmix} $r(xc(y))=r(x)y$, $cr(x)=x+\sigma_G^*\overline{\sigma_X^*}x$ and $rc(y)=2y$ for $x\in K^*_G(X)$ and $y\in KR^*_G(X)$, where $K^*_G(X)$ is extended to a~$\mathbb{Z}_8$-graded algebra by Bott periodicity.
		\end{enumerate}
	\end{enumerate}
\end{proposition}

\begin{proof}
(\ref{krcoeffring}) is given in \cite[Section 2]{Se}. The proof of (\ref{realcomplex}) is the same as in the nonequivariant case, which is given in \cite{At}.
\end{proof}

\begin{definition}\label{QuatK}
A Quaternionic $G$-vector bundle over a Real space $X$ is a complex vector bundle $E$ equipped with an anti-linear vector bundle endomorphism $J$ on $E$ such that $J^2=-\text{Id}_E$ and $J(g\cdot v)=
\sigma_G(g)\cdot J(v)$. Let $KH^*_G(X)$ be the corresponding $K$-theory constructed using Quaternionic~$G$-bundles over $X$.
\end{definition}
By generalizing the discussion preceding \cite[Lemma 5.2]{Se} to the equivariant and graded setting, we define
a natural transformation
\[t: \ KH_G^{-q}(X)\to KR^{-q-4}_G(X)\]
which sends
\[0\longrightarrow E_1\stackrel{f}{\longrightarrow} E_2\longrightarrow 0\]
to
\[0\longrightarrow \pi^*(\mathbb{H}\otimes_\mathbb{C} E_1)\stackrel{g}
{\longrightarrow}\pi^*(\mathbb{H}\otimes_\mathbb{C}E_2)\longrightarrow 0.\]
Here
\begin{enumerate}
	\item $E_i$, $i=1, 2$ are equivariant Quaternionic vector bundles on $X\times\mathbb{R}^{0, q}$ equipped with the Quaternionic structures $J_{E_i}$,
	\item $f$ is an equivariant Quaternionic vector bundle homomorphism which is an isomorphism outside $X\times\{0\}$,
	\item $\pi: X\times\mathbb{R}^{0, q+4}\to X\times\mathbb{R}^{0, q}$ is the projection map,
	\item $\mathbb{H}\otimes_\mathbb{C}E_i$ is the equivariant Real vector bundles equipped with the Real structure $J\otimes J_{E_i}$,
	\item $g$ is an equivariant Real vector bundle homomorphism def\/ined by $g(v, w\otimes e)=(v, vw\otimes f(e))$.
\end{enumerate}

The discussion in the last section of \cite{AS} can be extended to the equivariant setting and yields
\begin{proposition}\label{shiftbyfour}
$t$ is an isomorphism.
\end{proposition}

\subsection{The Real representation ring $RR(G)$}
In this Section, we review the basics of Real representations elaborated in \cite{F}. 
\begin{definition}\label{RRRing}
	\begin{enumerate}
		\item  A Real representation $V$ of $G$ is a finite-dimensional complex representation of $G$ equipped with an anti-linear involution $\sigma_V$ such that $\sigma_V(g\cdot v)=\sigma_G(g)\cdot\sigma_V(v)$. Similarly a Quaternionic representation is one equipped with an anti-linear endomorphism $J_V$ such that $J_V^2=-\text{Id}_V$ and $J_V(g\cdot v)=\sigma_G(g)\cdot J_V(v)$. For $\mathbb{F}=\mathbb{R}$ or $\mathbb{H}$, a morphism between $V$ and $W\in \mathcal{R}ep_\mathbb{F}(G)$ is a linear transformation from $V$ to $W$ which commutes with $G$ and respects both $\sigma_V$ and $\sigma_W$. Let $\mathcal{R}ep_\mathbb{R}(G)$ (resp. $\mathcal{R}ep_\mathbb{H}(G)$) be the category of Real (resp. Quaternionic) representations of $G$. The Real (resp. Quaternionic) representation group of $G$, denoted by $RR(G)$ (resp. $RH(G)$) is the Grothendieck group of $\mathcal{R}ep_\mathbb{R}(G)$ (resp. $\mathcal{R}ep_\mathbb{H}(G)$). $RR(G)$ is also a ring with tensor product as ring multiplication. 
		\item  Let $V$ be an irreducible Real (resp. Quaternionic) representation of $G$. Its \emph{commuting field} is defined to be $\text{Hom}_G(V, V)^{\sigma_V}$, which is isomorphic to either $\mathbb{R}$, $\mathbb{C}$ or $\mathbb{H}$. Let $RR(G, \mathbb{F})$ (resp. $RH(G, \mathbb{F})$) be the abelian group generated by the isomorphism classes of irreducible Real (resp. Quaternionic) representations with $\mathbb{F}$ as the commuting field.
		\item Let $R(G, \mathbb{C})$ be the abelian group generated by the isomorphism classes of those irreducible complex representations $V$ satisfying $V\ncong\sigma_G^*\overline{V}$. 
		\item Let $V$ be a $G$-representation. We use $\sigma_G^*V$ to denote the~$G$-representation with the same underlying vector space where the~$G$-action is twisted by $\sigma_G$, i.e. $\rho_{\sigma_G^*V}(g)v=\rho_V(\sigma_G(g))v$. We will use $\overline{\sigma_G^*}$ to denote the map on $R(G)$ def\/ined by $[V]\mapsto [\sigma_G^*\overline{V}]$.
	\end{enumerate}
\end{definition}
The following Proposition gives characterizations of Real and Quaternionic representations of various types. See Proposition \ref{OsheaSjamaar} (\ref{RRRRHR}) for alternative characterizations for $RR(G, \mathbb{R})$ and $RH(G, \mathbb{R})$.
\begin{proposition}\label{RealQuatStr}\cite[Propositions 2.16 and 2.18]{F}
\begin{enumerate}
	\item Let~$V$ be an irreducible Real representation of $G$.
	\begin{enumerate}\itemsep=0pt
		\item\label{realtype} The commuting field of~$V$ is isomorphic to $\mathbb{R}$ iff~$V$ is an irreducible complex representation and there exists $f\in\text{Hom}_G(V, \sigma_G^*\overline{V})$ such that $f^2=\text{Id}_V$.
		\item The commuting field of~$V$ is isomorphic to $\mathbb{C}$ iff $V\cong W\oplus\sigma_G^*\overline{W}$ as complex~$G$-representa\-tions, where $W$ is an irreducible complex~$G$-representation and $W\ncong \sigma_G^*\overline{W}$, and $\sigma_V(w_1, w_2)$ $=(w_2, w_1)$.
		\item\label{quattype} The commuting field of~$V$ is isomorphic to $\mathbb{H}$ iff $V\cong W\oplus \sigma_G^*\overline{W}$ as complex~$G$-repre\-sen\-ta\-tions, where $W$ is an irreducible complex~$G$-representation and there exists
$f\in\text{Hom}_G(W,\sigma_G^*\overline{W})$ such that $f^2=-\text{Id}_W$, and $\sigma_V(w_1, w_2)=(w_2, w_1)$.
	\end{enumerate}
	\item Let $V$ be an irreducible Quaternionic representation of $G$.
	\begin{enumerate}\itemsep=0pt
		\item The commuting field of $V$ is isomorphic to $\mathbb{H}$ iff $V\cong W\oplus\sigma_G^*\overline{W}$ as complex $G$-representations, where~$W$ is an irreducible complex $G$-representation and there exists $f\in\text{Hom}_G(W, \sigma_G^*\overline{W})$ such that $f^2=\text{Id}_W$ with $J(w_1, w_2)=(-w_2, w_1)$.
		\item The commuting field of $V$ is isomorphic to $\mathbb{C}$ iff $V\cong W\oplus\sigma_G^*\overline{W}$ as complex $G$-representations, where~$W$ is an irreducible complex $G$-representation and $W\ncong\sigma_G^*\overline{W}$ with $J(w_1, w_2)=(-w_2, w_1)$. 
		\item The commuting field of $V$ is isomorphic to $\mathbb{R}$ iff $V$ is an irreducible complex $G$-representation and there exists $f\in\text{Hom}_G(V, \sigma_G^*\overline{V})$ such that $f^2=-\text{Id}_V$.
	\end{enumerate}
\end{enumerate}	
\end{proposition}

\begin{proposition}\label{KR_G}
	\begin{enumerate}
		\item \cite[Proposition 3.3]{F}
	The map 
	\begin{align*}
f: & (RR(G, \mathbb{R})\oplus RH(G, \mathbb{R}))\otimes KR^*(\text{pt})\oplus r(R(G, \mathbb{C})\otimes K^*(+))\to
KR^*_G(\text{pt}),\\
&\rho_1\otimes x_1\oplus \rho_2\otimes x_2\oplus r(\rho_3\otimes \beta^i)\mapsto \rho_1\cdot x_1+\rho_2\cdot
x_2+r(\rho_3\otimes\beta^i)
	\end{align*}
	is an isomorphism of graded rings. Here elements in $RR(G, \mathbb{R})$ and $RH(G, \mathbb{R})$ are assigned with degree 0 and $-4$ respectively. 
		\item\label{realquatexch} We have that 
		\begin{align*}
			&\text{if }\rho\in RR(G, \mathbb{R})\subseteq KR^G_0(\text{pt}),\ \text{then }\rho\cdot\mu=\rho+\rho\in RH(G, \mathbb{H})\subseteq KR^G_{-4}(\text{pt}).\\
			&\text{If }\rho\in RH(G, \mathbb{R})\subseteq KR^G_{-4}(\text{pt}),\ \text{then }\rho\cdot\mu=\rho+\rho\in RR(G, \mathbb{H})\subseteq KR^G_0(\text{pt}).\\
			&\text{If }\rho\in R(G, \mathbb{C})\subseteq K^G_0(\text{pt}),\ \text{then }r(\rho)=\rho+\sigma_G^*\overline{\rho}\in r(R(G, \mathbb{C}))=RR(G, \mathbb{C})\subseteq KR^G_0(\text{pt}),\\
			&\text{and }r(\rho\cdot\beta^2)=\rho+\sigma_G^*\overline{\rho}\in r(R(G, \mathbb{C})\cdot\beta^2)=RH(G, \mathbb{C})\subseteq KR^G_{-4}(\text{pt}).
		\end{align*}
		The above representations $\rho+\rho$ and $\rho+\sigma_G^*\overline{\rho}$ should be equipped with the appropriate Real or Quaternionic structure specified by Proposition \ref{RealQuatStr}.
		\item Let 
		\begin{align*}
			\widetilde{R}&:=RR(G, \mathbb{R})\oplus RH(G, \mathbb{R})\oplus RR(G, \mathbb{C})\\
					   &=(RR(G, \mathbb{R})\oplus RH(G, \mathbb{R}))\otimes KR_0(\text{pt})\oplus r(R(G, \mathbb{C}))\\
					   &\subseteq KR^G_0(\text{pt})\oplus KR^G_{-4}(\text{pt})
		\end{align*}
		and $c: \widetilde{R}\to R(G)\cong K^G_0(\text{pt})$ be the map forgetting Real or Quaternionic structures. Then $c$ is injective and $\text{im}(c)=R(G)^{\overline{\sigma_G^*}}$. In particular, if $RR(G, \mathbb{C})=0$, $c$ is a ring isomorphism.
	\end{enumerate}
\end{proposition}
\begin{proof}
	For (2) use Proposition \ref{RealQuatStr}. For (3) use Proposition \ref{RealQuatStr} and \cite[Propositions 2.17, 2.24]{F}.
\end{proof}
\end{appendix}

\footnotesize\textsc{Department of Mathematics, Cornell University, Ithaca, NY 14853, USA\\
\\
National Center for Theoretical Sciences, Mathematics Division, National Taiwan University, Taipei 10617, Taiwan\\
\\
School of Mathematical Sciences and Institute for Geometry and its Applications, the University of Adelaide, Adelaide, SA 5005, Australia\\
\\
E-mail: }\texttt{chi-kwong.fok@adelaide.edu.au}\\
\textsc{URL: }\texttt{https://sites.google.com/site/alexckfok}
\end{document}